\documentclass[10pt, reqno]{amsart}
\pdfoutput=1

\usepackage{amsthm, amssymb, amsmath}
\usepackage{verbatim}
\usepackage{calrsfs}
\usepackage{graphicx}
\usepackage{epsfig}
\usepackage[all]{xy}

\newtheorem{thm}{Theorem}[section]
\newtheorem{lem}[thm]{Lemma}
\newtheorem{cor}[thm]{Corollary}

\newtheorem*{conjecture*}{Conjecture}

\theoremstyle{remark} 
\newtheorem*{question*}{Question}

\newtheorem{remark}[thm]{Remark}
\newtheorem{example}[thm]{Example}

\theoremstyle{definition} 
\newtheorem{define}[thm]{Definition}

\numberwithin{equation}{section}
\numberwithin{figure}{section}

\newcommand{\ZZ}{\mathbb{Z}}     
\newcommand{\RR}{\mathbb{R}}     
\newcommand{\Aff}{\mathbb{A}}      
\newcommand{\QQ}{\mathbb{Q}}      
\newcommand{\CC}{\mathbb{C}}      

\newcommand{\be}{\begin{equation}}
\newcommand{\ee}{\end{equation}}
\newcommand{\benn}{\begin{equation*}}
\newcommand{\eenn}{\end{equation*}}
\newcommand{\ba}{\begin{aligned}}
\newcommand{\ea}{\end{aligned}}
\newcommand{\bbm}{\begin{bmatrix}}
\newcommand{\ebm}{\end{bmatrix}}
\newcommand{\bpm}{\begin{pmatrix}}
\newcommand{\epm}{\end{pmatrix}}
\newcommand{\bi}{\begin{itemize}}
\newcommand{\ei}{\end{itemize}}
 
\newcommand{\Hom}{\operatorname{Hom}}
\newcommand{\im}{\operatorname{im}} 

\newcommand{\Div}{\operatorname{div}}    
\newcommand{\Pic}{\operatorname{Pic}^0}        
\newcommand{\Zh}{\operatorname{Zh}}   

\newcommand{\simarrow}{\stackrel{\sim}{\rightarrow}}    
\newcommand{\ip}[2]{\left\langle #1,#2 \right\rangle} 
\newcommand{\into}{\hookrightarrow}     

\newcommand{\Prin}{\operatorname{Prin}}
\renewcommand{\Div}{\operatorname{Div}}

\renewcommand{\Aff}{\operatorname{Aff}}
\newcommand{\Euclidean}{\operatorname{L^2}}

\title{Metric Properties of the Tropical Abel-Jacobi Map}

\author{Matthew Baker}
\address{
School of Mathematics \\
Georgia Institute of Technology \\
Atlanta, GA \\ 
USA} 
\email{mbaker@math.gatech.edu}
\urladdr{http://www.math.gatech.edu/~mbaker/}

\author{Xander Faber}
\address{
Department of Mathematics and Statistics \\
McGill University \\
Montr\'eal, QC  \\ 
CANADA} 
\email{xander@math.mcgill.ca}
\urladdr{http://www.math.mcgill.ca/xander/}

\subjclass[2000]{14H40 (primary);
 05C50, 05C38 (secondary)
}
\keywords{Tropical Curve, Tropical Jacobian, Picard Group, Abel-Jacobi, Metric Graph, Foster's Theorem}

\begin{document}
	\begin{abstract}

Let $\Gamma$ be a tropical curve (or metric graph), and fix a base 
point $p \in \Gamma$.	
We define the Jacobian group $J(G)$ of a finite weighted graph $G$, and show that the Jacobian $J(\Gamma)$ is canonically isomorphic to the direct limit of $J(G)$ over all weighted graph models $G$ for $\Gamma$.  This result is useful for reducing certain questions about the Abel-Jacobi map 
$\Phi_p : \Gamma \to J(\Gamma)$, defined 
by Mikhalkin and Zharkov,
to purely combinatorial questions about weighted graphs.
We prove that $J(G)$ is finite if and only if the edges in each $2$-connected component of $G$ are commensurable over $\QQ$.  As an application of our direct limit theorem, we derive some local comparison formulas between $\rho$  and $\Phi_p^*(\rho)$ for three different natural ``metrics'' $\rho$ on $J(\Gamma)$.  One of these
formulas implies that $\Phi_p$ is a tropical isometry when $\Gamma$ is $2$-edge-connected.  Another shows that the canonical measure
$\mu_{\Zh}$ on a metric graph $\Gamma$, defined by S. Zhang,
measures lengths on $\Phi_p(\Gamma)$ with 
respect to the ``sup-norm'' on $J(\Gamma)$.  
	\end{abstract}

\maketitle

\section{Introduction}

In \cite{Bacher_et_al_1997, Baker_Norine_2007, Balacheff_Invariant_2006, Kotani_Sunada_2000,
Mikhalkin_Zharkov_Tropical_Jacobians_2008} and other works, combinatorial analogues of classical facts about Riemann surfaces and their Jacobians are explored in the context of graphs and tropical curves.  (For the purposes of this article, a tropical curve is the same
thing as a compact metric graph of finite total length.)
The Jacobian of a finite (unweighted) graph $G$, as defined in 
\cite{Bacher_et_al_1997}, 
is a certain finite Abelian group whose cardinality is the number of
spanning trees in $G$.  On the other hand, the Jacobian of a tropical curve $\Gamma$ of genus $g$ is a real torus of dimension $g$.
In both cases, given a base point $p$, there is a natural map from $G$ (resp. $\Gamma$) to its Jacobian, called the {\em Abel-Jacobi map}, 
which is injective if and only if
the graph (resp. tropical curve) contains no bridges.  (A \textbf{bridge} is an edge that is not contained in any non-trivial cycle.)
These are combinatorial counterparts of the classical fact that the Abel-Jacobi map from a Riemann surface $X$ to its Jacobian is injective if and only if the genus of $X$ is at least $1$.
In the present paper, we explore some additional parallels between 
these algebro-geometric and combinatorial theories, and
present some new combinatorial results with no obvious analogues in algebraic geometry.

The first part of the paper is devoted to constructing a rigorous framework for understanding tropical curves and their Jacobians via {\em models} (cf.~\cite{Baker_Faber_2006}).
A model for a tropical curve $\Gamma$ is 
just a weighted graph $G$ whose geometric realization is $\Gamma$.
In order to work with tropical curves and their Jacobians from this point of view, one needs to be able to pass freely between different models for $\Gamma$, so it is desirable to set up a general theory of Jacobians for weighted graphs which corresponds to the usual notion of $J(G)$ from \cite{Bacher_et_al_1997,Baker_Norine_2007} when all edge lengths are equal to $1$.  
Essentially every desired consequence works out exactly as one would hope: for each weighted graph $G$ there is a canonical isomorphism 
$J(G) \cong \Pic(G)$ (generalizing the graph-theoretic version of Abel's theorem from \cite{Bacher_et_al_1997}), and $J(\Gamma)$ is canonically
the direct limit of $J(G)$ over all models $G$ for $\Gamma$.  This allows us to give a very simple proof of the tropical version of Abel's theorem, 
first proved in \cite{Mikhalkin_Zharkov_Tropical_Jacobians_2008}.  Our direct limit point of view --- thinking of tropical curves as limits of models --- 
reduces various 
questions about tropical curves and their Jacobians to questions about finite weighted graphs.  
This discussion will occupy sections 2--4.

The Jacobian of a weighted graph $G$ is sometimes finite and sometimes infinite; it is natural to wonder when each case occurs.  We provide a complete answer to this question in \S\ref{Sec: Commensurable}: $J(G)$ is finite if and only if the edges in each $2$-connected component of $G$ are commensurable over $\QQ$.  The proof is an application of potential theory on metric graphs (cf.~\cite{Baker_Faber_2006,Baker_Rumely_2007}). At the other extreme, if the lengths of the edges in each maximal 2-connected component are $\QQ$-linearly independent, then the Jacobian group is free Abelian of rank $\#V(G) - \#Br(G) - 1$, where $V(G)$ and $Br(G)$ are the sets of vertices and bridges of $G$, respectively.  If one fixes the combinatorial type of $G$ (i.e., the underlying unweighted graph), then it seems like a difficult question to describe the possible group structures for $J(G)$ as one varies the edge lengths. 

The original motivation for the present work was to better understand the canonical measure $\mu_{\Zh}$ on a metric graph $\Gamma$ of genus at least $1$ defined by Zhang in \cite{Zhang_1993}.  The measure $\mu_{\Zh}$ plays a role in Zhang's theory analogous to the role played in Arakelov theory 
by the canonical $(1,1)$-form $\omega_{X}$ on a Riemann surface $X$ of genus at least $1$. (For the definition, see \S\ref{Sec: Zhang Measure}.) 
One of the most important descriptions of the $(1,1)$-form $\omega_X$ is that it is obtained by pulling back the flat Riemannian metric on the Jacobian $J(X)$ under the Abel-Jacobi map $\Phi_p : X \into J(X)$, for any choice of base point $p \in X$.
It is natural to wonder whether Zhang's measure has a similar description in terms of the Abel-Jacobi map $\Phi_p : \Gamma \to J(\Gamma)$ from a tropical curve $\Gamma$ to its Jacobian.  Although the situation does not appear to be fully analogous to the classical theory, we provide such a description in
this paper: 
the measure $\mu_{\Zh}$ can be obtained by pulling back a canonical metric on $J(\Gamma)$ which we call the ``sup-norm'' or ``Foster/Zhang'' metric.  More precisely, if $e$ is an edge of some model for $\Gamma$, then the length of $\Phi_p(e)$ with respect to the sup-norm metric is $\mu_{\Zh}(e)$.  This gives a quantitative version of the fact that the edges of $\Gamma$ contracted to a point by $\Phi_p$ are precisely the bridges (which by Zhang's explicit formula for $\mu_{\Zh}$ are exactly the segments of $\Gamma$ to which $\mu_{\Zh}$ assigns no mass).  This analogy is described in greater detail in \S\ref{Sec: Zhang Measure}.

There are at least three natural ``metrics'' on $J(\Gamma)$ for which one can explicitly compute the length of $\Phi_p(e)$ on $J(\Gamma)$ in terms of the length $\ell(e)$ of an edge $e$ in some model for $\Gamma$.  In addition to the sup-norm metric, one can also define a ``Euclidean metric'' on $J(\Gamma)$, and we prove that the length of $\Phi_p(e)$ in the Euclidean metric is $\sqrt{\ell(e) \mu_{\Zh}(e)}$.  Although this formula is not as clean as the formula for the sup-norm, it is striking that there is a simple answer in both cases.
A third metric on $J(\Gamma)$ is the ``tropical metric'':
we prove that away from the bridge edges, $\Phi_p$ is a {\em tropical isometry} from $\Gamma$ onto its image.
This result is the most natural of our three metric comparison theorems from the point of view of tropical geometry, whereas our
sup-norm theorem is arguably the most relevant one from the Arakelov-theoretic point of view.
We define and discuss these metric structures in \S\ref{Sec: Metrics}.

In the final section, we interpret the numbers $F(e) = \mu_{\Zh}(e)$ 
(which we call the {\em Foster coefficients}, after R. Foster
\cite{Foster_Theorem_1949})
in terms of electrical network theory, orthogonal projections, weighted spanning trees, and random walks on graphs.

Unlike some authors, we have chosen in this paper to consider only tropical curves with finite total length. This has no real effect on the generality of our discussion of Jacobians; because infinite-length segments do not support harmonic 1-forms, they play no role in the construction of the Jacobian and are collapsed under the Abel-Jacobi map.

\textbf{Acknowledgments.} During the writing of this article, the first author was supported in part by NSF grant DMS-0600027, and the second author by the Centre de Recherches Math\'ematiques and the Institut des Sciences Math\'ematiques in Montreal. The authors would like to thank Farbod Shokrieh for helpful feedback on this work. They also thank the anonymous referees for many thoughtful suggestions.

\section{The Weighted Jacobian}

	The goal of this section is to define and investigate the Jacobian group of a weighted graph. We prove the Jacobian is canonically isomorphic to the Picard group, and that Jacobian groups behave well with respect to length-preserving subdivision of edges.  We also determine exactly when the Jacobian of a weighted graph is a finite group.

\subsection{Weighted graphs}
\label{Sec: Wtd graphs}

A \textbf{weighted graph} $G$ in this paper will be an edge-weighted, connected multigraph, possibly with loop edges, endowed with
a fixed orientation of the edges.  (Most of our constructions, including the Jacobian and Picard group of $G$, are independent of the choice of orientation.) 
More precisely, $G$ is given by specifying a nonempty vertex set $V(G)$, an edge set $E(G)$, a length map\footnote{
	In our geometric study of graphs we have chosen to use lengths rather than the more 
	standard notion of weights $w: E(G) \to \RR_{>0}$ defined by $w(e) = \ell(e)^{-1}$.
} 
$\ell: E(G) \to \RR_{>0}$, and an edge assignment map $\iota: E(G) \to V(G) \times V(G)$ such that for any pair of distinct vertices $x$ and $x'$, there exists a sequence of vertices $x=x_0, x_1, \ldots, x_n=x'$ and a sequence of edges $e_1, \ldots, e_n$ such that $\iota(e_i) = (x_{i-1}, x_i)$ or $\iota(e_i) = (x_i, x_{i-1})$. For an edge $e$, the \textbf{tail vertex} $e^-$ and the \textbf{head vertex} $e^+$ are defined by $\iota(e) = (e^-, e^+)$. Two vertices $x$ and $y$ are \textbf{adjacent}, written $x \sim y$, if there is an edge $e$ such that $\iota(e) = (x,y)$ or $\iota(e) = (y,x)$. In either case, we write $x \in e$ to indicate that $x$ is one of the vertices of $e$. 

Let $A$ be either the ring of integers $\ZZ$ or the field $\RR$. The free $A$-module generated by the vertex set $V(G)$ (resp. the edge set $E(G)$) is called the \textbf{module of $0$-chains (resp. $1$-chains) with coefficients in $A$}, and is denoted by $C_0(G, A)$ (resp. $C_1(G, A)$). Note that $C_j(G, \ZZ) \subset C_j(G, \RR)$ for $j=0,1$. Since $C_0(G, A)$ is canonically isomorphic to its dual $\Hom(C_0(G, A), A)$, we may identify a $0$-chain $f = \sum_{x \in V(G)} n_x. x$ with the function $f: V(G) \to A$ given by $f(x) = n_x$. A similar remark applies to $1$-chains. 

Some authors follow a more canonical approach, 
replacing each edge $e$ with two edges $e$ and $\bar{e}$ corresponding to an edge with two possible orientations. A 1-chain is then a function on the edge space such that $f(\bar{e}) = -f(e)$. We have chosen instead to fix an orientation, in the interest of making integration along 1-chains look more like the standard treatment in Riemannian geometry. 

We may define inner products on $C_0(G, \RR)$ and $C_1(G, \RR)$ by
	\begin{eqnarray*}
		 	\ip{f_1}{f_2} & = & \sum_{x \in V(G)} f_1(x)f_2(x),
				\hspace{0.45in} f_1, f_2 \in C_0(G, \RR) \\
			\ip{\alpha_1}{\alpha_2} & = & \sum_{e \in E(G)} \alpha_1(e)\alpha_2(e)\ell(e),
				\hspace{0.2in} \alpha_1, \alpha_2 \in C_1(G, \RR).
	\end{eqnarray*}
The differential operator $d:C_0(G,\RR) \to  C_1(G,\RR)$ applied to a $0$-chain $f \in C_0(G, \RR)$ gives the ``slope of  $f$ along the edge $e$'': 
	\[
		(df)(e) = \frac{f(e^+)-f(e^-)}{\ell(e)}. 
	\]
The adjoint operator $d^*:C_1(G,\RR) \to C_0(G,\RR)$ acts on a $1$-chain $\alpha$ by	
	\begin{equation*} \label{Eq: Adjoint}
		(d^*\alpha)(x) = \sum_{\substack{e \in E(G) \\ e^+ = x}} \alpha(e) - 
			\sum_{\substack{e \in E(G) \\ e^- = x}} \alpha(e).
	\end{equation*}
With these definitions, one immediately checks that the matrix of $\Delta = d^*d$ relative to the basis of $C_0(G,\RR)$ given by the vertices of $G$ is equal to the usual weighted Laplacian matrix of $G$. (See \cite[\S5]{Baker_Faber_2006} for the definition.)

We define $H_1(G, \RR) = \ker(d^*)$ and $H_1(G, \ZZ) = \ker(d^*) \cap C_1(G, \ZZ)$.
These will be called the \textbf{real (resp. integral) $1$-cycles}. By general linear algebra, we get a canonical orthogonal ``Hodge decomposition''
	\begin{equation} \label{Eqn: Chain Decomposition}
		C_1(G,\RR) = \ker(d^*) \oplus \im(d) = H_1(G, \RR) \oplus \im(d).
	\end{equation}

A \textbf{1-form on $G$} is an element of the real vector space with formal basis $\{de: e \in E(G)\}$. A 1-form $\omega = \sum \omega_e .de$ is \textbf{harmonic} if 
	\begin{equation} \label{Eq: Omega}
		\sum_{\substack{e \in E(G) \\ e^+ = x}}  \omega_e =
			\sum_{\substack{e \in E(G) \\ e^- = x}} \omega_e \qquad \text{ for all $x \in V(G)$.}
	\end{equation}
Write $\Omega(G)$ for the space of harmonic $1$-forms. Define integration of the basic $1$-form $de$ along an edge $e'$ by
	\[
		\int_{e'}  de = \begin{cases} \ell(e) & \text{if $e = e'$} \\ 0 & \text{if $e \not= e'$.} \end{cases}
	\]
By linearity, we can extend this definition to obtain an integration pairing:
	\begin{eqnarray*}
		\Omega(G) \times C_1(G, \RR) &\longrightarrow& \RR \\
			(\omega, \alpha) \hspace{1.3cm}&\mapsto& \int_\alpha \omega.
	\end{eqnarray*}

\begin{lem} \label{Lem: Perfect Pairing}
	The kernel on the left of the integration pairing is trivial, while the kernel on the right is $\im(d)$. In particular, 
	integration restricted to $\Omega(G) \times H_1(G, \RR)$ gives a perfect pairing --- i.e., an isomorphism
	 $H_1(G, \RR) \simarrow \Omega(G)^*$.
\end{lem}

\begin{proof}
	For each $e \in E(G)$, let $s_e \in \RR$ be a real number. It follows immediately from the definitions that the $1$-form $\omega = \sum s_e.de$ is harmonic if and only if the $1$-chain $\alpha = \sum s_e.e$ is a $1$-cycle. As $\int_\alpha \omega = \sum s_e^2 \ell(e) = 0$ if and only if $s_e = 0$ for all $e$, we find that the integration pairing restricted to $\Omega(G) \times H_1(G, \RR)$ is perfect. 
	
	To finish the proof, it suffices by \eqref{Eqn: Chain Decomposition} to show that $\int_{df}  \omega = 0$ for all $0$-chains $f \in C_0(G, \RR)$ and all $\omega \in \Omega(G)$. If we write $\omega = \sum \omega_e. de$, then
	\begin{align*}
		\int_{df} \omega = \sum_{e \in E(G)} \sum_{e' \in E(G)}
				&\frac{f(e^+) - f(e^-)}{\ell(e)} \ \omega_{e'} \int_e de' 
			= \sum_{e \in E(G)} \left[f(e^+) - f(e^-)\right]\omega_e \\
			&= \sum_{x \in V(G)} f(x) \left(\sum_{\substack{e \in E(G) \\ e^+ = x}} \omega_e
				-\sum_{\substack{e \in E(G) \\ e^- = x}} \omega_e \right).
	\end{align*}
The final expression vanishes by the definition of harmonic form. 
\end{proof}

	For each $1$-chain $\alpha \in C_1(G, \RR)$, we can define an integration functional $\int_\alpha: \Omega(G) \to \RR$ as above. Lemma~\ref{Lem: Perfect Pairing} shows every element of $\Omega(G)^*$ arises in this way. Consider the following subgroup given by integration on integral $1$-chains:
	\[
		\Omega(G)^{\sharp} = \left\{\int_\alpha \in \Omega(G)^* : \alpha \in C_1(G, \ZZ)\right\}.	
	\]
Again referring to Lemma~\ref{Lem: Perfect Pairing}, we may identify the group of integral cycles $H_1(G, \ZZ)$ with a subgroup of $\Omega(G)^\sharp$ via $\alpha \mapsto \int_\alpha$. 

\begin{define}
	The \textbf{Jacobian} of a weighted graph $G$ is given by
		\[
			J(G) = \Omega(G)^\sharp / H_1(G, \ZZ).
		\]
\end{define}

From the construction of $J(G)$ one sees that it is a finitely generated Abelian group. A canonical set of generators is given by functionals of the form $\int_e$ for $e \in E(G)$.

 \begin{remark}
There are several other definitions of the Jacobian group in the literature that deserve comment:
\begin{enumerate}
	\item  A definition of Jacobian group for unweighted graphs was given in \cite{Bacher_et_al_1997}, which establishes several basic properties of $J(G)$, including the Abel-Jacobi isomorphism.  In the final section, the authors suggest ``There would be no difficulty to extend all previous considerations to graphs with positive weights on vertices and edges.'' However, their definition of Jacobian group does not immediately generalize well to the weighted case because $H_1(G, \ZZ)$ is not an integral lattice.  Our definition not only agrees with theirs in the unweighted case, but it also behaves well with respect to length-preserving edge subdivisions and passage to the limit over all such subdivisions. Nevertheless, we would like to acknowledge the inspiration gained from \cite{Bacher_et_al_1997}.
	
	\item For a weighted graph $G$, define $J(G)_\RR := \Omega(G)^* / H_1(G, \ZZ)$ equipped with the inner product structure arising by duality from Lemma~\ref{Lem: Perfect Pairing}. The ``jacobienne'' in \cite{Balacheff_Invariant_2006} and the ``Albanese torus'' with its ``flat structure'' in \cite{Kotani_Sunada_2000} are both defined to be the real torus $H_1(G, \RR) / H_1(G, \ZZ)$ with the inner product defined in \S\ref{Sec: Wtd graphs}. They are canonically isomorphic to $J(G)_\RR$, which contains our Jacobian group $J(G)$ as a subgroup. The ``Jacobian torus'' in \cite{Kotani_Sunada_2000} is dual to the Albanese torus.


\end{enumerate}
\end{remark}

\begin{example}
	Let $G$ be the weighted graph with two vertices $x,y$, two edges $e_1, e_2$ such that $\iota(e_1) = \iota(e_2) = (x,y)$, and $\ell(e_1) = 1 - \ell(e_2) = r < 1$. Then $H_1(G, \ZZ) = \ZZ.(e_1 - e_2)$. If we identify $H_1(G, \RR)$ with $\Omega^*(G)$ as in Lemma~\ref{Lem: Perfect Pairing}, then $\int_{e_1} = r.(e_1 - e_2)$ and $\int_{e_2} = (r-1).(e_1 - e_2)$. It follows that 
	\begin{align*}
		J(G) &= \Big( \ZZ. r(e_1 - e_2) + \ZZ.(1-r)(e_1 - e_2) \Big) \; \Big/ \; \ZZ.(e_1 - e_2) \\
		&\cong \begin{cases} \ZZ & \text{ if $r \not\in \QQ$,} \\
			\ZZ / n\ZZ & \text{ if $r = m/n$ and $\gcd(m,n) =1$.}
			\end{cases}
	\end{align*}
\end{example}

\begin{example}
	Extending the previous example, let $G$ be the weighted graph with $n+1$ vertices and $n+1$ edges arranged in a single directed cycle. Let $\ell_0, \ldots, \ell_n$ be the lengths of the edges, and suppose that $\sum \ell_j = 1$. Then one can verify that
	\[
		J(G) \cong \left(\ZZ. \ell_0 + \cdots + \ZZ. \ell_n \right)  /  \ \ZZ \ \subset \  \RR / \ZZ.
	\]
If $\ell_0, \ldots, \ell_n$ are rational, then $J(G)$ is a (finite) torsion group. At the other extreme, if $\ell_0, \ldots, \ell_n$ are $\QQ$-linearly independent, then $J(G)$ is free Abelian of rank~$n$.  Compare with Theorems~\ref{Thm: Commensurate} and~\ref{Thm: Maximal Rank}. 
\end{example}



\subsection{The Picard Group and the Jacobian}
\label{Sec: Picard Group}	
	Let $G$ be a weighted graph. Recall that we have a homomorphism $d: C_0(G, \RR) \to C_1(G, \RR)$ and its adjoint $d^*: C_1(G, \RR) \to C_0(G, \RR)$. Define $\im(d)_\ZZ = \im(d) \cap C_1(G, \ZZ)$. Elements of $\im(d)_\ZZ$ correspond to functions $f: V(G) \to \RR$ with integer slopes, modulo constant functions. 
	
	The \textbf{divisor group of $G$}, denoted $\Div(G)$, is defined to be $C_0(G, \ZZ)$. The \textbf{degree} of a divisor $D = \sum_{x \in V(G)} D(x).x$ is $\deg(D) = \sum D(x)$.  Define the following subgroups of $\Div(G)$:
	\begin{eqnarray*}
		\Div^0(G) &=& \{D \in \Div(G) : \deg(D) = 0\} \\
		\Prin(G) &=& d^*\left(\im(d)_\ZZ\right).
	\end{eqnarray*}
A simple calculation shows that if $\alpha \in C_1(G, \RR)$ is any $1$-chain, then $\deg(d^*\alpha) = 0$. Consequently, we find $\Prin(G) \subset \Div^0(G)$.

\begin{define}
	The \textbf{(degree zero) Picard Group} of a weighted graph $G$ is given by
		\[
			\Pic(G) = \Div^0(G) / \Prin(G).
		\]
\end{define}

	We want to show that the homomorphism $d^*: C_1(G, \ZZ) \to C_0(G, \ZZ) = \Div(G)$ descends to an isomorphism $h: J(G) \to \Pic(G)$. To that end, consider the following diagram of homomorphisms (the map $h$
will be discussed in the theorem below):
	\begin{equation} \label{Eq: Comm Diag}
			\xymatrix{
				& & 0 \ar[d] & 0 \ar[d] & \\
				 0 \ar[r] & H_1(G, \ZZ) \ar@{=}[d] \ar[r] & {H_1(G, \ZZ) \oplus \im(d)_\ZZ}
				 	\ar[r]^{\phantom{aaaaaa} d^*}
				 \ar[d] & \Prin(G) \ar[r] \ar[d]  & 0 \\
				 0 \ar[r]& H_1(G, \ZZ) \ar[r] & C_1(G, \ZZ) \ar[r]^{d^*} 
				 	\ar[d] & \Div^0(G) \ar[r] \ar[d] & 0 \\
				& & J(G) \ar@{-->}[r]^h \ar[d] & \Pic(G) \ar[d] & \\
				& & 0 & 0 & \\
				}
	\end{equation}	
	
\begin{lem} \label{Lem: Comm Diag}
	Let $G$ be a weighted graph. Then the diagram \eqref{Eq: Comm Diag} (without the map $h$) is commutative and exact. 
\end{lem}

\begin{proof}
	This is an exercise in unwinding the definitions. Note that the map $C_1(G, \ZZ) \to J(G)$ is defined by $\alpha \mapsto \int_\alpha$. 
\end{proof}

\begin{thm} \label{Thm: Jac/ Pic}
	Let $G$ be a weighted graph. There exists a unique homomorphism $h: J(G) \to \Pic(G)$ that makes the diagram \eqref{Eq: Comm Diag} (with the map $h$) commutative. Moreover, $h: J(G) \to \Pic(G)$ is an isomorphism. 
\end{thm}

\begin{proof}
Apply the Snake Lemma to \eqref{Eq: Comm Diag}.
\end{proof}


	For unweighted graphs, a different proof of this theorem is given in \cite[Prop.~7]{Bacher_et_al_1997}.



\subsection{Compatibility under refinement}

Given two weighted graphs $G_1$ and $G_2$, we say that $G_2$ \textit{refines} $G_1$ if there exist an injection $a: V(G_1) \hookrightarrow V(G_2)$ and a surjection $b: E(G_2) \twoheadrightarrow E(G_1)$ such that for any edge $e$ of $G_1$ there exist vertices $a(e^-)=x_0, x_1, \ldots, x_n = a(e^+) \in V(G_2)$ and edges $e_1, \ldots, e_n \in E(G_2)$ satisfying (1) $b^{-1}(e) = \{e_1, \ldots, e_n\}$, (2) $\sum_{i=1}^n \ell(e_i) = \ell(e)$, and (3) $\iota_2(e_i) = (x_{i-1}, x_i)$ for each $i=1, \ldots n$. (Here $\iota_2$ is the associated edge map for $G_2$.) See Figure~\ref{Fig: Refinement} below for an example. Roughly, one refines a weighted graph by subdividing its edges in a length preserving fashion and endowing the new edges with the induced orientation. In the interest of readable exposition, we will identify the vertices of $G_1$ with their images in $V(G_2)$ by $a$, and we will say that the edge $e$ of $G_1$ is subdivided into $e_1, \ldots, e_n$ if $b^{-1}(e) = \{e_1, \
 \ldots, e_n\}$. Henceforth, we will suppress mention of $a$ and $b$ when we speak of refinements.

If $G_2$ is a refinement of $G_1$, we write $G_1 \leq G_2$. This is a partial ordering on the collection of all weighted graphs. There is a canonical refinement homomorphism $r_{21}: C_1(G_1, \RR) \to C_1(G_2,\RR)$ defined by $r_{21}(e) = \sum_{i=1}^n e_i$ if $e$ is an edge of $G_1$ that is subdivided into $e_1, \ldots, e_n$. Identifying a $1$-chain $\alpha \in C_1(G_1, \RR)$ with the corresponding function on the edges of $G_1$, we have $r_{21}(\alpha)(e_i) = \alpha(e)$ for all $i=1, \ldots, n$. If $G_1 \leq G_2 \leq G_3$, then it is clear that we have $r_{31} = r_{32}   r_{21}$. If no confusion will arise, we avoid writing the subscripts on the refinement homomorphisms.

The refinement homomorphism $r: C_1(G_1, \RR) \to C_1(G_2, \RR)$ induces a push-forward homomorphism $r_*$ on $1$-forms defined by $r_*(de) = \sum_{i=1}^n de_i$, in the notation of the last paragraph. 

\begin{lem} \label{Lem: Refinement}
	Let $G_1$ and $G_2$ be two weighted graphs such that $G_2$ refines $G_1$. 	
	\begin{enumerate}
	
	\item The refinement homomorphism $r$ induces an isomorphism of real vector spaces
		$H_1(G_1,\RR) \simarrow H_1(G_2, \RR)$ such that 
			$r(H_1(G_1, \ZZ)) = H_1(G_2, \ZZ)$. Moreover, if $d_i: C_0(G_i, \RR) \to C_1(G_i, \RR)$
			denotes the differential operator on $G_i$, then 
			$r\left( \im(d_1)_\ZZ \right) \subset \im(d_2)_\ZZ$.
	\item	 The push-forward homomorphism $r_*$ induces an isomorphism 
	$\Omega(G_1) \simarrow \Omega(G_2)$, and it is compatible with the integration pairing in
	the sense that
	for each $\alpha \in C_1(G, \RR)$ and $\omega \in \Omega(G_1)$, we have
			$\int_{r(\alpha)} r_*(\omega) = \int_\alpha \omega$.
			
	\item The refinement homomorphism induces an isomorphism $\Omega(G_1)^* \simarrow \Omega(G_2)^*$
	given by $\int_\alpha \mapsto \int_{r(\alpha)}$ for $\alpha \in H_1(G_1, \RR)$. In particular, it induces an
	isomorphism
		\[
			\Omega(G_1)^* / H_1(G_1, \ZZ) \simarrow \Omega(G_2)^* / H_1(G_2, \ZZ).
		\]
	\end{enumerate}
\end{lem}

\begin{proof}
By induction on the number of vertices in the refinement, we may assume that $G_2$ is obtained from $G_1$ by subdividing an edge $e_0$ into $e_1$ and $e_2$.  Let $\iota_1(e_0) = (x_0, x_2)$. Let the new vertex $x_1$ on $G_2$ be defined by $\iota_2(e_1) = (x_0, x_1)$ and $\iota_2(e_2) = (x_1, x_2)$.  (See Figure~\ref{Fig: Refinement}.)

\begin{figure}[ht]
	\begin{center}
		\begin{picture}(200, 30)(0,0)
			\put(-60,10){\includegraphics{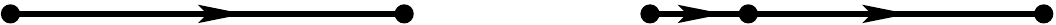}}
			\put(-60,-2){$e^-_0$}
			\put(53,-2){$e^+_0$}
			\put(-2,22){$e_0$}
			\put(135,22){$e_1$}
			\put(188,22){$e_2$}
			\put(124,-2){$x_0$}
			\put(153,-2){$x_1$}
			\put(238,-2){$x_2$}
		\end{picture}
	\end{center}
	\caption{An illustration of the edge $e_0$ subdivided into edges $e_1, e_2$, with arrows indicating
		the agreement of the orientations. }
	\label{Fig: Refinement}
\end{figure}
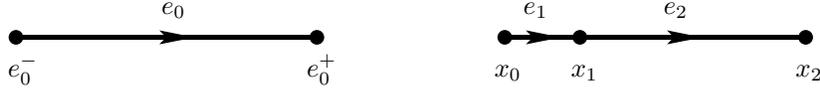
	
We now begin the proof of (1). Define $d_1^*$ and $d_2^*$ to be the adjoints of the corresponding differential operators on $G_1$ and $G_2$, respectively. Let $\alpha \in C_1(G_1, \RR)$. If $x\not=x_0, x_1, x_2$, then the definition of $(d_1^*\alpha)(x)$ and $(d_2^*r(\alpha))(x)$ doesn't involve $e_0, e_1$ or $e_2$. Hence $(d_2^*r(\alpha))(x) =(d_1^*\alpha)(x)$. For the remaining vertices, we see that
	\be \label{Eq: Refined cycles}
		\ba
			(d_2^*r(\alpha))(x_0) &= \sum_{\substack{e \in E(G_2) \\ 
			e^+ = x_0}} r(\alpha)(e) - 
			\sum_{\substack{e \not= e_1 \in E(G_2) \\ e^- = x_0}} r(\alpha)(e) 
			\ -r(\alpha)(e_1) \\
			&= \sum_{\substack{e \in E(G_1) \\ e^+ = x_0}} \alpha(e) - 
			\sum_{\substack{e \not= e_0 \in E(G_1) \\ e^- = x_0}} \alpha(e) 
			\ - \alpha(e_0) \\
			&= (d_1^*\alpha)(x_0), \\
			(d_2^*r(\alpha))(x_1) &= r(\alpha)(e_1) - r(\alpha)(e_2) = 
			\alpha(e_0) -\alpha(e_0) = 0.
		\ea
	\ee
The computation for the vertex $x_2$, similar to that for $x_0$, shows $(d_2^*r(\alpha))(x_2) = (d_1^*\alpha)(x_2)$. As $H_1(G, \RR) = \ker(d^*)$ for any weighted graph $G$, these computations imply that $r\left(H_1(G_1, \RR)\right) \subset H_1(G_2, \RR)$. 

For the opposite inclusion, take $\beta \in H_1(G_2, \RR)$. Since $d_2^*\beta(x_1) = \beta(e_1) - \beta(e_2) = 0$, we may define a cycle $\alpha$ on $G_1$ by setting $\alpha(e_0) = \beta(e_1) = \beta(e_2)$ and $\alpha(e) = \beta(e)$ for all other edges $e$. By construction $r(\alpha) = \beta$, hence $r(H_1(G_1, \RR)) = H_1(G_2, \RR)$. As $r$ is injective, it induces an isomorphism between cycle spaces as desired. 

From the fact that $r(\alpha)$ is integer valued if and only if $\alpha$ is integer valued, we deduce $r(H_1(G_1, \ZZ)) \subset H_1(G_2,\ZZ)$. The opposite inclusion follows by the computation given above for $\RR$-valued cycles. 

Now suppose $f \in C_0(G_1, \RR)$, and define $g \in C_0(G_2, \RR)$ by
	\[
		g(x) = \begin{cases}
			f(x) & \text{if $x \not= x_1$} \\
			& \\
			\displaystyle \frac{f(x_2)\ell(e_1) + f(x_0)\ell(e_2)}{\ell(e_1) + \ell(e_2)} & \text{if $x = x_1$}.
		\end{cases}
	\]
Then a direct computation shows $d_2g = r\left(d_1f\right)$. Hence $r \left( \im(d_1) \right) \subset \im(d_2)$. But $r$ preserves integrality of $1$-chains, so we have proved the final claim of (1).

The computation \eqref{Eq: Refined cycles} shows $r_*$ induces an isomorphism as desired in assertion~(2). Indeed, the defining relation for cycles ($d^*\alpha = 0$) is identical to that for harmonic $1$-forms as in 
\eqref{Eq: Omega}. To check compatibility of the integration pairing, it suffices by linearity to check $\int_{r(e)} r_*(de') = \int_e de'$ for any pair of edges $e, e' \in E(G_1)$. If $e \not= e'$, then both integrals vanish. If $e = e' \not= e_0$, then $r(e) = e$ and $r_*(de') = de'$, so the equality of integrals is obvious in this case as well. Finally, suppose $e = e' = e_0$. Then
	\[
		\int_{r(e_0)} r_*(de_0) = \int_{e_1 + e_2} \left(de_1 + de_2\right) = \ell(e_1) + \ell(e_2) = \ell(e_0)
			= \int_{e_0} de_0.
	\]

Finally, the proof of (3) is given by composing the duality isomorphism from Lemma~\ref{Lem: Perfect Pairing} with
the isomorphism in part (1):
	\[
		\Omega(G_1)^* \simarrow H_1(G_1, \RR) \simarrow H_1(G_2, \RR) \simarrow \Omega(G_2)^*.
	\]
The middle isomorphism preserves integral cycles, which finishes assertion~(3).
\end{proof}

\begin{thm} \label{Thm: compatibility}
	If $G_1$ and $G_2$ are weighted graphs such that $G_2$ refines $G_1$, then the refinement homomorphism $r$ 
	descends to a canonical injective homomorphism $\rho: J(G_1) \to J(G_2)$. 
	If $G_3$ is a weighted graph that refines $G_2$, so that $G_1 \leq G_2 \leq G_3$, then the injections
	$\rho_{ij}: J(G_j) \to J(G_i)$ satisfy $\rho_{31} = \rho_{32}\rho_{21}$. 
\end{thm}

\begin{proof} 
	By Lemmas~\ref{Lem: Comm Diag} and~\ref{Lem: Refinement}(1), the refinement map $r$ induces a homomorphism
	\begin{align*}
		J(G_1) \cong C_1(G_1, \ZZ) / & \left(H_1(G_1, \ZZ) \oplus \im(d_1)_\ZZ\right) \\
			& \hspace{1cm} \longrightarrow
			C_1(G_2, \ZZ) / \left(H_1(G_2, \ZZ) \oplus \im(d_2)_\ZZ\right) \cong J(G_2).
	\end{align*}
The remaining assertions of the theorem are left to the reader. 
\end{proof}

	Fix a weighted graph $G_0$ and define $R(G_0)$ to be the set of all weighted graphs that admit a common refinement with $G_0$. That is, $G \in R(G_0)$ if there exists a weighted graph $G'$ such that $G \leq G'$ and $G_0 \leq G'$. Then $R(G_0)$ is a directed set under refinement, and the previous theorem shows that $\{J(G) : G \in R(G_0)\}$ is a directed system of groups. We now describe the structure of the ``limit Jacobian:''\footnote{A special case of Theorem~\ref{Thm: Limit Jacobian} has been proved independently by Haase, Musiker, and Yu \cite{HMY_2009}.}
	
\begin{thm} \label{Thm: Limit Jacobian}
	Let $G_0$ be a weighted graph. There is a canonical isomorphism
		\[
			\lim_{\substack{\longrightarrow \\ G \in R(G_0)}} J(G) \simarrow \Omega(G_0)^* / H_1(G_0, \ZZ).
		\]
\end{thm}

Before proving the theorem, let us recall a few definitions and facts regarding cycles from algebraic graph theory. (See \cite{Biggs_Potential_Theory_1997}.) A \textbf{path} in the weighted graph $G$ is a sequence $x_0, e_1, x_1, e_2, \ldots, x_n$ of alternate vertices and edges such that $\{e_i^-, e_i^+\} = \{x_{i-1}, x_i\}$.\footnote{In \cite{Biggs_Potential_Theory_1997}, this is called a \textbf{walk}. There it is called a \textit{path} if $e_{j-1} \not= e_j$ for $j=2, \ldots, n$.} That is, the vertices of the edge $e_i$ are precisely $x_{i-1}$ and $x_i$, but perhaps not in the order dictated by the orientation. A path is \textbf{closed} if $x_0 = x_n$. Given a path $P$ and an edge $e$, we can associate an integer $\epsilon(P, e)$ that is the number of times the sequence $e^-, e, e^+$ occurs in the path $P$ less the number of times the sequence $e^+, e, e^-$ occurs in $P$. Intuitively, it is the number of times the path~$P$ traverses the edge $e$ counted with signs depending on the orientation of~$e$. Define a $1$-chain associated to the path $P$ by $\alpha_P = \sum_e \epsilon(P, e).e$. If $P$ is a closed path, then $\alpha_P$ is a $1$-cycle.

Let $T$ be a spanning tree in $G$. Given any edge $e$ in the complement of $T$, we can define a 
$1$-cycle associated to it as follows. There is a unique minimal path in $T$ beginning at $e^+$ and ending at $e^-$. If we preface this path with $e^-, e$, we get a closed path $P$. Note that $|\epsilon(P, e)| \leq 1$ for all edges $e$. We call $\alpha_P$ (as defined above) the \textbf{fundamental cycle} associated to $T$ and $e$, and denote it by $\alpha_{T, e}$. For a fixed spanning tree $T$, the cycles $\{\alpha_{T, e}: e \in E(G) \smallsetminus E(T)\}$ form a basis of $H_1(G, \ZZ)$.

\begin{proof}[Proof of Theorem~\ref{Thm: Limit Jacobian}]
First note that when $G$ varies over $R(G_0)$, a natural directed system is formed by the groups $\Omega(G)^* / H_1(G, \ZZ)$ and the isomorphisms of Lemma~\ref{Lem: Refinement}(3).  Since $\Omega(G)^\sharp \subset \Omega(G)^*$ for any weighted graph~$G$, it follows that there is also a canonical injective homomorphism $J(G) \hookrightarrow \Omega(G)^* / H_1(G, \ZZ)$.
Direct limit is an exact functor, so the maps $J(G) \hookrightarrow \Omega(G)^* / H_1(G, \ZZ)$ induce a canonical injective homomorphism
	\be \label{Eq: can injection}
		\lim_{\longrightarrow} J(G) \hookrightarrow \lim_{\longrightarrow} \Omega(G)^* / H_1(G, \ZZ),
	\ee
where the limits are over all weighted graphs $G \in R(G_0)$. 

Choose a spanning tree $T$ for $G_0$, and enumerate the edges in the complement of $T$ by $e_1, \ldots, e_g$. Write $\alpha_j = \alpha_{T, e_j}$ for the $j^{\operatorname{th}}$ fundamental cycle associated to $T$. The fundamental cycles $\alpha_1, \ldots, \alpha_g$ form a basis for $H_1(G_0, \ZZ)$, and for any refinement $G$ of $G_0$, the cycles $r(\alpha_1), \ldots, r(\alpha_g)$ form a basis for $H_1(G, \ZZ)$, where $r: C_1(G_0, \ZZ) \to C_1(G, \ZZ)$ is the refinement homomorphism. To prove surjectivity of the map in \eqref{Eq: can injection}, it suffices to show that for any real numbers $t_1, \ldots, t_g \in [0,1)$, there is a refinement $G$ of $G_0$ and a 1-chain $\beta \in C_1(G,\ZZ)$ such that $\int_\beta = \int_{\sum_j t_j r(\alpha_{T, e_j})}$. 

Define a basis for $\Omega(G_0)$ as follows. If $\alpha_j = \sum_e \alpha_j(e).e$, set $\omega_j = \sum_e \alpha_j(e).de$.  Now for each $j = 1, \ldots, g$, choose an integer $n_j$ and a positive real number $u_j < \ell(e_j)$ so that 
	\[
		n_ju_j = \sum_{i=1}^g t_i \int_{\alpha_i}  \omega_j .
	\]
Subdivide the edge $e_j$ into two edges $e_j'$ and $e_j''$ such that 
	\begin{itemize}
		\item $p_j$ is the new vertex with $\iota(e_j') = (e_j^-, p_j)$ and $\iota(e_j'') = (p_j, e_j^+)$, and
		\item $\ell(e_j') = u_j$ and $\ell(e_j'') = \ell(e_j) - u_j$.
	\end{itemize}
Let $G$ be the refinement of $G_0$ given by adjoining all of the vertices $p_1, \ldots, p_g$ to $V(G_0)$ and by subdividing the edges $e_1, \ldots, e_g$ in the manner described above. Finally, set $\beta = \sum n_i e_i'$.

As the edge $e_j$ appears in the cycle $\alpha_{i}$ only if $i=j$, we see that 
	\begin{equation*}
		\int_\beta r_*(\omega_j) = \sum_i n_i \int_{e_i'} r_*(\omega_j) = n_j u_j  
			= \sum_i t_i \int_{\alpha_i}  \omega_j = \int_{\sum_i t_i  r\left(\alpha_i\right)} r_*(\omega_j).
	\end{equation*}
The final equality occurs because $r$ and $r_*$ are compatible with the integration pairing. Since this chain of equalities holds for each $j$, and since $r_*(\omega_1), \ldots, r_*(\omega_g)$ is a basis for $\Omega(G)^*$, we conclude that $\int_\beta = \int_{\sum_j t_j r(\alpha_{j})}$.

Finally, we note that the direct limit on the right of \eqref{Eq: can injection} is canonically isomorphic to any particular term in the limit. In particular, it is isomorphic to $\Omega(G_0)^* / H_1(G_0,\ZZ)$, which completes the proof.
\end{proof}

	
\subsection{Properties of the Weighted Jacobian}
\label{Sec: Commensurable}

	In this section we deduce a few basic properties of the structure of weighted Jacobians. Our main result 
	states that the Jacobian of a weighted graph is finite if and only if the edge lengths in each maximal 2-connected 
	subgraph are commensurable. 
	
\begin{thm} \label{Thm: Edge Scaling}
	Let $G$ be a weighted graph, and let $G'$ be the weighted graph obtained from $G$ by scaling all edge lengths
	simultaneously by some positive real number. Then $J(G)$ and $J(G')$ are canonically isomorphic. 
\end{thm}

\begin{proof}
	Since the Jacobian group is isomorphic to the Picard group, it suffices to prove the result for Picard groups. 
Identify the vertex and edge sets of $G$ and $G'$. Then one verifies that $\Div^0(G) = \Div^0(G')$ and $\Prin(G) = \Prin(G')$. 
\end{proof}

\begin{thm} \label{Thm: Direct Product}
	Suppose $G$ is a weighted graph that can be written as the pointed sum of two subgraphs $G_1$ and $G_2$.
	Then there is a canonical isomorphism $J(G) \simarrow J(G_1) \times J(G_2)$.
\end{thm}

\begin{proof}
	Again, we will work with the Picard group rather than the Jacobian. Let $G_1 \cap G_2 = \{p\}$, so
	that $p$ is the separating vertex. Then $\Div^0(G)$ is canonically isomorphic to 
	$\Div^{0}(G_1) \times \Div^0(G_2)$ using the following observation:
		\[
			\sum_{x \in V(G)} D(x).x = \sum_{x \in V(G)} D(x).(x-p) =
				\sum_{x \in V(G_1)} D(x).(x-p) + \sum_{x \in V(G_2)} D(x).(x-p).
		\]
	Moreover, $\Prin(G)$ splits according to this decomposition as well. 
	It follows immediately that $\Pic(G) \simarrow \Pic(G_1) \times \Pic(G_2)$.
\end{proof}

A weighted graph will be called \textbf{2-connected} if its geometric realization is topologically 2-connected; i.e., if it cannot be disconnected by deleting a single point.\footnote{For unweighted graphs, our definition of 2-connectivity differs from the convention adopted by many graph theorists in the special case where $|V(G)|=2$.}    
Note that 2-connectivity is stable under refinement. 
An inductive argument on the number of vertices shows that any weighted graph can be decomposed as a union of its maximal $2$-connected subgraphs and of its bridges. Combining Theorem~\ref{Thm: Direct Product} with the fact that the Jacobian of a tree is trivial yields

\begin{cor} \label{Cor: Decompose Product}
	Let $G$ be a weighted graph with maximal $2$-connected subgraphs $G_1, \ldots, G_s$.
	Then there is a canonical isomorphism $J(G) \simarrow J(G_1) \times \cdots \times J(G_s)$.
\end{cor}

	Let $a_1,\ldots,a_n$ be positive real numbers.  Then $a_1,\ldots,a_n$ are called \textbf{commensurable} if
	 $a_i / a_j \in \QQ$ for all indices $1 \leq i,j \leq n$. This condition is equivalent to saying there exists a 
	 positive real number $t$ and natural numbers $k_1, \ldots, k_n$ such that $a_i = t k_i$ for each 
	 $i=1, \ldots, n$. 
	 We will say that a collection of edges
	 of a weighted graph is \textbf{commensurable} if their lengths are commensurable. 

\begin{thm} \label{Thm: Commensurate}
If $G$ is a weighted graph, then the Jacobian group $J(G)$ is finite if and only if the edges of each maximal 2-connected subgraph of $G$ are commensurable. 
\end{thm}

By way of contrast, we also have the following result:

\begin{thm}
\label{Thm: Maximal Rank}
	Let $G$ be a weighted graph. Suppose that the edge lengths in each maximal 2-connected component of $G$ are $\QQ$-linearly independent. Then $J(G)$ is free Abelian of rank $\#V(G) - \#Br(G) - 1$, where $Br(G)$ is the set of bridges of $G$. 
\end{thm}

Before giving the proofs, we need a lemma which may be of independent interest.  If $\Gamma$ is a metric (a.k.a.~metrized) graph, recall that the fundamental potential kernel $j_z(x,y)$ is characterized by
 the differential equation $\Delta_x j_z(x,y) = \delta_y(x) - \delta_z(x)$ subject to the initial condition $j_z(z,y)=0$
(see, for example, \cite{Baker_Faber_2006}).
It is a basic fact that $j_z(x,y) \geq 0$ and $j_z(x,y)=j_z(y,x)$ for all $x,y,z \in \Gamma$.

\begin{lem} \label{Lem: 2-connected}
A metric graph $\Gamma$ is $2$-connected if and only if 
$j_z(x,y) > 0$ for all $x,y,z \in \Gamma$ with $x,y \neq z$.
More specifically, $x,y$ belong to the same connected component of
$\Gamma - z$ if and only if $j_z(x,y) > 0$.
\end{lem}

\begin{proof}
Suppose that $j_z(x,y) = 0$ with $x,y \neq z$. 
Then $y \neq x$, since $j_z(x,x)$ is the effective resistance between the distinct points $x$ and $z$, which cannot be zero. 
As a function of~$t$, $j_z(x,t)$ is harmonic outside $x$ and $z$, and so is harmonic at $t=y$ and achieves its minimum value of zero there.  By harmonicity, it must be zero 
in a neighborhood of $y$. If $x,y$ belong to the same connected component of $\Gamma - z$, then we can iterate this argument at each point along a path from $y$ to $x$ that does not pass through $z$ in order to conclude that $j_z(x,x) = 0$. 
As already mentioned, this is a contradiction.
(The same argument proves the \textbf{maximum principle} for harmonic functions on a metric graph: a non-constant harmonic function on a connected open set $U \subseteq \Gamma$ cannot achieve its maximum or minimum value on $U$.)

Conversely, if $x,y$ belong to distinct connected components of $\Gamma - z$, then the maximum principle implies that $j_z(x,y)=0$,
since the maximum value of the harmonic function $j_z(x,t)$ on the closure of the connected component $U$ of $\Gamma - z$ containing $y$ must occur at the unique boundary point $z$ of $U$, where the function is zero.
\end{proof}

\begin{proof}[Proof of Theorem~\ref{Thm: Commensurate}]

Suppose first that $J(G)$ is a finite group. By Corollary~\ref{Cor: Decompose Product} the Jacobian of each maximal 2-connected subgraph of $G$ is also finite, and so we may replace $G$ by any one of these subgraphs without loss of generality. 

As $J(G)$ and the Picard group are isomorphic, we see $\Pic(G)$ is finite. We will show that the quotient of the lengths of any two edges that share a common vertex is a rational number; as our graph is connected, this suffices to show all edges are commensurable. Note that since $G$ is $2$-connected, there are no loop edges. 

Let $e$ be an edge with vertices $x,z$, and let $e'$ be an edge with vertices $y, z$. 
Since $\Pic(G)$ is finite, it follows that the divisor $x-z$ has finite order in the Picard group. This means that for some natural number $m$ and some function $f: V(G) \to \RR$ with $df \in \im(d)_\ZZ$, we have
$m(x-z) = d^*df$. We may further suppose that $f(z) = 0$, since replacing $f$ by $f - f(z)$ does not affect $d^*df$. If $\Delta$ is the Laplacian operator on the associated metric graph $\Gamma$ 
(i.e., the realization of $G$), and if we extend $f$ by linearity to $\Gamma$, this means that $\Delta_t(f(t)/m) = \delta_x(t) - \delta_z(t)$, where $\delta_p$ is the point measure at $p$. The fundamental potential kernel $j_{z}(x,t)$ satisfies this same differential equation and also vanishes at $z$, so that $f(t)/m = j_z(x,t)$. Similarly, since $y-z$ is torsion in $\Pic(G)$, there exists a positive integer $n$ and a function $g: V(G) \to \RR$ with $g(z) = 0$ and $dg \in \im(d)_\ZZ$ such that $g(t)/n = j_z(y,t)$. But now we can use the symmetry of the $j$-function to see that
	\[
		\frac{\ell(e) \cdot \text{integer}}{n} = \frac{g(x)}{n} = j_z(y,x) = j_z(x,y) = \frac{f(y)}{m}
			= \frac{\ell(e') \cdot \text{integer}}{m}.
	\]
The appearance of the unspecified integers in the numerators is due to the fact that $f$ and $g$ have integer slopes. By 
Lemma~\ref{Lem: 2-connected}, $j_z(x,y) \not= 0$, from which we conclude that $\ell(e) / \ell(e')$ is rational.

For the converse direction, Corollary~\ref{Cor: Decompose Product} 
reduces us to showing that if $G$ is 2-connected and the edges of $G$ are commensurable, then $J(G)$ is finite. As we observed above, there exists a real number $t$ and a natural number $k_e$ for each edge $e$ such that $\ell(e) = tk_e $ for all edges $e \in E(G)$. Let $G'$ be the graph obtained from $G$ by subdividing the edge $e$ into $k_e$ edges of length $t$. Then $G'$ is a refinement of $G$ with equal edge lengths. Theorem~\ref{Thm: Edge Scaling} shows that we do not change the structure of the Jacobian by assuming that $t=1$.
Now $J(G) \hookrightarrow J(G')$ by Theorem~\ref{Thm: compatibility}, and $\#J(G') = \#\Pic(G')$ is equal to the number of spanning trees of $G'$, since $G'$ has all edge lengths equal to 
$1$ \cite[Prop.~1(iii)]{Bacher_et_al_1997}.  Thus the order of $J(G)$ divides the number of spanning trees of $G'$, and in particular it has finite cardinality.
\end{proof}

From the proof of Theorem~\ref{Thm: Commensurate}, we obtain

\begin{cor}
Let $G$ be a weighted graph for which $J(G)$ is finite.
Then there exists a refinement $G'$ of $G$ such that all edge lengths in each maximal 2-connected subgraph of $G'$ have equal length. Moreover, the order of $J(G)$ divides the number of (unweighted) spanning trees of $G'$. 
\end{cor}

\begin{proof}[Proof of Theorem~\ref{Thm: Maximal Rank}]
	Corollary~\ref{Cor: Decompose Product} allows us to immediately reduce to the case that $G$ is 2-connected, in which we must show that $J(G)$ is free Abelian of rank $\#V(G) - 1$. 
	
	We again work with the Picard group. Observe that if we fix a vertex $y \in V(G)$, then any divisor of degree~0 can be written uniquely as
	\[
		D = \sum_{x \in V(G)} D(x).x = \sum_{x \in V(G)} D(x). \left(x - y\right).
	\]
Thus $\Div^0(G)$ is freely generated by the set of divisors $\{x - y : x \neq y \}$, which has cardinality $\#V(G) - 1$. 

	It suffices to prove that $\Prin(G) = 0$. This is trivial if $G$ consists of a single vertex. Otherwise, suppose $f:V(G) \to \RR$ is a function with integer slopes. Let $\gamma$ be a nontrivial closed path in $G$ with no backtracking and such that no edge is traversed twice. Let $x_0, x_1, \ldots, x_n = x_0$ be the vertices of the path $\gamma$, write $\ell_{i,j}$ for the length of the edge of $\gamma$ passing from $x_i$ to $x_j$, and write $m_{i,j} \in \ZZ$ for the slope of $f$ along this same edge, oriented from $x_i$ to $x_j$. Then 
	\begin{eqnarray*}
		f(x_1) &=& f(x_0) + \ell_{0,1} \cdot m_{0,1} \\
		f(x_2) &=& f(x_1) + \ell_{1,2} \cdot m_{1,2} \\
			& \vdots & \\
		f(x_n) & = & f(x_{n-1}) + \ell_{n-1, n} \cdot m_{n-1, n}.
	\end{eqnarray*}
Substituting each equation into the next gives the linear dependence relation 
	\[
		m_{0,1} \cdot \ell_{0,1} + m_{1,2} \cdot \ell_{1,2} + \cdots + m_{n-1, n} \cdot \ell_{n-1,n} = 0.
	\]
The lengths $\ell_{i,j}$ are $\QQ$-linearly independent, so all of the slopes $m_{i,j}$ must vanish. As $G$ has no bridges, we can find such a path $\gamma$ passing through any edge of $G$. That is, $f$ is constant on $G$, and $d^*d(f) = 0$. 
\end{proof}


\section{The Tropical Jacobian}
\label{Sec: Tropical Jacobian} 

Let $\Gamma$ be a \textbf{metric graph}, by which we will mean a compact connected metric space such that each $p \in \Gamma$ has a neighborhood $U_p$ isometric to a star-shaped set of valence $n_p \geq 1$, endowed with the path metric. A \textbf{star-shaped set of valence $n_p$} is a set of the form
	\be \label{Eq: Star-shaped set}
		S(n_p, r_p) =\{z \in \CC: z=te^{k\cdot 2\pi i / n_p} 
			\text{ for some $0 \leq t < r_p$ and some $k \in \ZZ$}\}.
	\ee
We will also assume that $\Gamma$ is equipped with an oriented simplicial decomposition in order to facilitate the computation of homology. 

There is a one-to-one correspondence between metric graphs on one hand, and (abstract) compact tropical curves on the other. (See \cite[\S3.3]{Mikhalkin_Zharkov_Tropical_Jacobians_2008}.) We may therefore speak of these objects interchangeably, and we have chosen to work in the language of tropical curves in this paper. Throughout, all tropical curves will be implicitly assumed compact. 
	 
We can think of a tropical curve as a geometric realization of a weighted graph. A weighted graph $G$ yields a tropical curve $\Gamma$ by associating to an oriented edge $e \in E(G)$ the line segment $[0, \ell(e)]$, and then by gluing these line segments according to the edge assignment map $\iota: E(G) \to V(G) \times V(G)$.  If $G'$ is a refinement of~$G$, then it is clear that $G$ and $G'$ give rise to tropical curves $\Gamma$ and $\Gamma'$ which are canonically isometric. The weighted graph $G$ (and also $G'$) is called a \textbf{model} of~$\Gamma$. Conversely, given a tropical curve $\Gamma$, we can associate a model $G$ by choosing a nonempty vertex set $V \subset \Gamma$ containing all points of $\Gamma$ of valence different from~$2$. We note that the preferred simplicial decomposition of $\Gamma$ corresponds to fixing a preferred model $G_0$ for $\Gamma$. Every other model will have an orientation compatible with $G_0$. Note that the collection $R(G_0)$ of weighted graphs that admit a common refinement with $G_0$ is none other than the collection of models of $\Gamma$.

A \textbf{segment} of the tropical curve $\Gamma$ relative to a model $G$ is the geometric realization in $\Gamma$ of an edge of $G$. If we speak of a segment of $\Gamma$ without extra qualifier, we implicitly mean a segment relative to the preferred model $G_0$. We will often abuse notation by writing $e$ for both an edge of a model $G$ as well as for the corresponding segment of $\Gamma$.

	The space of \textbf{harmonic 1-forms} on  $\Gamma$ is defined as the direct limit
		\[	
			\Omega(\Gamma) = 
	   \lim_{\substack{\longrightarrow \\ G \in R(G_0)}} \Omega(G),
		\]
	where we recall that $\Omega(G)$ is canonically isomorphic to $\Omega(G')$ if $G'$ refines $G$. In other words, a harmonic $1$-form on $\Gamma$ can be described by giving a model $G$ for $\Gamma$ and a harmonic 1-form on $G$.\footnote{Our notion of harmonic 1-forms is equivalent to the notion of 1-forms used in \cite{Mikhalkin_Zharkov_Tropical_Jacobians_2008}. There a 1-form on a tropical curve is defined by giving an open cover $\{U_\alpha\}$ of $\Gamma$ by star-shaped sets and a continuous piecewise affine function $f_\alpha: U_\alpha \to \RR$ on each chart such that the sum of the slopes of $f_\alpha$ at each point of the chart is zero (taking orientation into account), and such that $f_\alpha$ differs from $f_\beta$ by a constant function on the overlap $U_\alpha \cap U_\beta$. Given a model $G$ for $\Gamma$, this definition implies that we can associate a well-defined real number $\omega_e$ to each edge $e$ of $G$ --- namely the slope of $f_\alpha$ --- in such a way that the sum of these slopes is zero at each vertex. Evidently this gives a harmonic 1-form $\omega = \sum \omega_e. de$ on $G$.}  Evidently $\Omega(\Gamma)$ is a real vector space, and $\Omega(\Gamma) \cong \Omega(G_0)$.

\begin{define}
	Let $\Gamma$ be a tropical curve with a given model $G_0$. The \textbf{Jacobian} of $\Gamma$ is given by
	\[
		J(\Gamma) 
		= \Omega(\Gamma)^* / H_1(\Gamma, \ZZ),
	\]
	where the homology group $H_1(\Gamma, \ZZ)$ is computed using the simplicial decomposition of 
	$\Gamma$ prescribed by the model $G_0$. The group $J(\Gamma)$ is given the quotient topology, where
	$\Omega(\Gamma)^*$ is equipped with the real vector space topology.
	
\end{define}

Choosing a basis for $H_1(\Gamma, \ZZ)$ shows that $J(\Gamma)$ is (non-canonically) isomorphic to the real torus $(\RR / \ZZ)^g$, where $g = \operatorname{rk} H_1(\Gamma,\ZZ) = \#E(G_0) - \#V(G_0) + 1$. 

\begin{remark} 
The choice of model $G_0$ is necessary for the definition of $J(\Gamma)$, but any two models yield canonically isomorphic Jacobians. 
\end{remark}

By the discussion immediately preceding the definition of $J(\Gamma)$ and by Theorem~\ref{Thm: Limit Jacobian}, we obtain

\begin{cor} \label{Cor: Can Iso 2}
	Let $\Gamma$ be a tropical curve equipped with a fixed model $G_0$. Then there is a canonical isomorphism
	\[
		\lim_{\substack{\longrightarrow \\ G \in R(G_0)}} J(G) \simarrow J(\Gamma)  ,
	\]
where $R(G_0)$ is the collection of weighted graphs that admit a common refinement with $G_0$. Moreover, for each $G \in R(G_0)$ the above isomorphism induces a canonical inclusion $J(G) \hookrightarrow J(\Gamma)$.
\end{cor} 

	One can also define a notion of \textbf{Picard group} for a tropical curve. We let $\Div^0(\Gamma)$ be the 
	free Abelian group of divisors of degree zero on $\Gamma$, and we let $\Prin(\Gamma)$ be the group of divisors
	that arise as the Laplacian of a continuous piecewise affine function on $\Gamma$ with integer slopes. Then the
	Picard group is the quotient $\Pic(\Gamma) = \Div^0(\Gamma) / \Prin(\Gamma)$. If we choose a model $G$
	for $\Gamma$, then the degree zero divisors supported on the vertices of $G$ generate the subgroup $\Pic(G)$ of $\Pic(\Gamma)$. In this way we get an inclusion $\Pic(G) \hookrightarrow \Pic(\Gamma)$, and we may view the
	 subgroup $\Pic(G)$ as ``$G$-rational points'' of the Picard group. Moreover, the commutative diagram
	in Lemma~\ref{Lem: Comm Diag} is compatible with refinements, so one can check that we obtain an 
	isomorphism 
		\[
			\lim_{\substack{\longrightarrow \\ G \in R(G_0)}} \Pic(G) \simarrow \Pic(\Gamma).
		\]
	In particular, combining these observations with 
	Theorem~\ref{Thm: Jac/ Pic} and Corollary~\ref{Cor: Can Iso 2},
	we obtain a new proof of the following result,
	originally proved in \cite[Thm.~6.3]{Mikhalkin_Zharkov_Tropical_Jacobians_2008} by a different method.
	
\begin{thm}[Tropical Abel-Jacobi theorem]
If $\Gamma$ is a tropical curve, there is a canonical isomorphism between $J(\Gamma)$ and $\Pic(\Gamma)$.
\end{thm}

	If $V$ is a real vector space of dimension $n$ and $\Lambda$ is a lattice in $V$, then $V/\Lambda$ is non-canonically isomorphic
	to a real torus $\left(\RR / \ZZ\right)^n$. 
	A \textbf{tropical structure} on $V / \Lambda$is defined to be a second lattice $\Lambda' \subset V$. 	
	The tangent space at a point of $V/\Lambda$ can be identified with the vector space $V$; now $\Lambda'$ can 
	be viewed as a lattice of sections in the tangent space $T(V / \Lambda)$. An element $\lambda \in \Lambda'$ is 
	called an \textbf{integral tangent vector}. Such a $\lambda$ is called \textbf{primitive} if it cannot
	be written as $\lambda = t\lambda'$ for some $\lambda' \in \Lambda'$ and some integer $t > 1$. 
	
	The lattice of \textbf{harmonic 1-forms with integer
	 coefficients} 
	on $\Gamma$, denoted $\Omega_\ZZ(\Gamma)$, is defined to be an element 
	$\sum \omega_e.de \in \Omega(\Gamma)$, where $\omega_e \in \ZZ$
	for all edges $e$ (of the implicitly chosen model). The \textbf{tropical structure} on the Jacobian 
	$J(\Gamma)$ is defined to be the dual lattice 
		\[
			\Omega_\ZZ(\Gamma)^* =  \left\{\int_\alpha \in \Omega(\Gamma)^* : \int_\alpha \omega \in \ZZ
			\text{ for all $\omega \in \Omega_\ZZ(\Gamma)$} \right\}.
		\]
		
	Now let $\Gamma$ be a tropical curve, and let $e$ be a segment of $\Gamma$. The interior of $e$ is 
	(canonically) isometric to the open interval $(0, \ell(e))$, and as such it is a smooth manifold of dimension~1. 
	The segment $e$ itself is a manifold with boundary, and at the
	two boundary points the tangent space of $e$ is an infinite ray.
	The metric on $\Gamma$ induces a canonical metric on the tangent space at each point of $e$. 
	We define a \textbf{tangent vector} on $\Gamma$ to be an element of $T_p(e)$ for some segment $e$ of $\Gamma$ and some
point $p \in e$.  
By abuse of notation, we will write $T_p(\Gamma)$ for the set of all tangent vectors at $p$, which may be viewed as the union of the infinite rays determined by the distinct segments emanating from $p$. We define an 
	\textbf{integral tangent vector} to be a tangent vector with integer length; an integral tangent vector is called 
	\textbf{primitive} if it has unit length.  
	
	Let us return to our general tropical torus $(V / \Lambda, \Lambda')$, and let $\Gamma$ be a tropical curve.
	A continuous map $f: \Gamma \to V/\Lambda$ is said to be \textbf{tropical} if the following three conditions hold: 
	\begin{itemize}
		\item[(i)] $f$ is piecewise affine; 
		\item[(ii)] At each point $p \in \Gamma$, the differential map 
			$Df_p:T_p(\Gamma) \to T_{f(p)}(V/\Lambda)$ 
			carries integral tangent vectors to integral tangent vectors; and 
		\item[(iii)] At each point $p \in \Gamma$, $f$ is \textbf{harmonic}. That is, 
			we have $\sum_{\lambda} Df_p(\lambda) = 0$, where  
			the sum is over all \textit{primitive} integral tangent vectors at $p$. 
	\end{itemize}
Conditions (i) and (ii) assert that  $f$ is locally given by $f(t) = v + t\lambda'$, where $t$ is a unit-speed coordinate on some segment of $\Gamma$, $\lambda' \in \Lambda'$, and $v \in V$. Condition (iii) is also referred to in the literature as a ``balancing condition'' at $p$.


\section{The Abel-Jacobi Map}

For the duration of this section, let $\Gamma$ be a tropical curve equipped with a preferred model $G_0$. Fix a point $p$ of $\Gamma$ (which need not be a vertex of $G_0$). We define the \textbf{Abel-Jacobi} map $\Phi_p: \Gamma \to J(\Gamma)$ as follows. For a point $q \in \Gamma$, choose a model $G \in R(G_0)$ containing $p$ and $q$ as vertices. Let $P$ be a path from $p$ to $q$ in $G$, and let $\alpha_P$ be the associated $1$-chain (as defined just before the proof of Theorem~\ref{Thm: Limit Jacobian}). Then $\int_{\alpha_P}$ defines an element of $J(G)$, and we set $\Phi_p(q)$ to be the image of $\int_{\alpha_P}$ in $J(\Gamma)$ under the canonical inclusion $J(G) \hookrightarrow J(\Gamma)$. (See Corollary~\ref{Cor: Can Iso 2}.) 

The map $\Phi_p$ depends on the basepoint $p$, but not on the choice of model $G$. Indeed, let $G_1$ and $G_2$ be models containing the vertices $p$ and $q$ and let $G_3$ be a common refinement. Let $P_1$ and $P_2$ be paths from $p$ to $q$ on $G_1$ and $G_2$, respectively. Let $r_{ij}$ be the refinement homomorphism from $G_j$ to $G_i$. Then $r_{31}(\alpha_{P_1})-r_{32}(\alpha_{P_2})$ is an integral $1$-cycle on $G_3$; i.e., $\int_{r_{31}(\alpha_{P_1})} \equiv \int_{r_{32}(\alpha_{P_2})}$ modulo $H_1(G_3, \ZZ)$.  As the integration pairing is compatible with the refinement homomorphisms (Lemma~\ref{Lem: Refinement}), it follows that $\Phi_p: \Gamma \to J(\Gamma)$ is well-defined. 

The next result is proved in \cite{Mikhalkin_Zharkov_Tropical_Jacobians_2008}, but we give an alternative proof using our discrete techniques.  We will push these methods just a little further in the next section to prove that the Abel-Jacobi map is a tropical isometry away from the segments that are collapsed.

\begin{thm} \label{Thm: Abel-Jacobi Properties}
	Let $p \in \Gamma$ be a basepoint. Then the Abel-Jacobi map 
	$\Phi_p: \Gamma \to J(\Gamma)$ satisfies the following properties:
	\begin{enumerate}
		\item $\Phi_p$ is continuous.
		\item The restriction of $\Phi_p$ to any simply-connected subgraph $\Gamma_1$ 
			containing $p$ can be factored 
			through the projection $\Omega(\Gamma)^* \to J(\Gamma)$:
				\benn
					\xymatrix{
						& \Omega(\Gamma)^* \ar[d] \\
						\Gamma_1 \ar@{-->}[ur]^{\widetilde{\Phi}_p} 
						\ar[r]_{\Phi_p \hspace{.2in} } & J(\Gamma).
					}
				\eenn
			\textup(The lift is unique if we require that $\widetilde{\Phi}_p(p) = 0$.\textup)
		\item If $e$ is a segment of $\Gamma$ such that the complement of the interior of $e$ is 
			disconnected, then $\Phi_p$ collapses $e$ to a point. 
		\item If $e$ is a segment of $\Gamma$ such that the complement of the interior of $e$ is 
			connected, then for any simply-connected subgraph $\Gamma_1$ containing $e$ and any lift 
			$\widetilde{\Phi}_p$ as in \textup{(2)}, the segment $e$ maps to a straight segment in 
			$\Omega(\Gamma)^*$. 
			
		\item Let $e$ be any segment of $\Gamma$. 
		Then
			\[
				\frac{1}{\ell(e)}\int_e \in \Omega_\ZZ(\Gamma)^*.
			\]
			If $\int_e \not= 0$ in $\Omega(\Gamma)^*$, then $\frac{1}{\ell(e)}\int_e$ is a primitive
			element of the lattice $\Omega_\ZZ(\Gamma)^*$. 
		
		\item The Abel-Jacobi map is tropical.
	\end{enumerate}
\end{thm}

The simple way to say (3) and (4) is that $\Phi_p$ collapses any segment that isn't part of a cycle, and it maps segments of cycles to straight segments. The only place where the image segment can change direction is at a point of $\Gamma$ of valence greater than two. 

\begin{remark}
	If $\Gamma$ is a tropical curve and $G$ is a model for $\Gamma$, then $J(G)$ is the subgroup of 
	$J(\Gamma)$ generated by the images of the vertices of $G$ under any Abel-Jacobi map $\Phi_p: \Gamma \to 
	J(\Gamma)$ with $p \in V(G)$. 
\end{remark}

For the proof, we will need the following graph-theoretic characterization of the basic $1$-chains of $G$ that give rise to trivial integration functionals. 

\begin{lem} \label{Lem: Connected vs. cycles}
	Let $G$ be a weighted graph and $e$ an edge of $G$. Then $\int_e = 0$ if and only if the graph $G-e$ obtained 	by deleting the edge $e$ is disconnected. 
\end{lem}

Farbod Shokrieh suggested the following proof, which simplifies our original one.

\begin{proof}
Observe that $e \in \im(d)$ if and only if $e$ is a bridge; i.e., if and only if $G-e$ is disconnected. Now apply Lemma~\ref{Lem: Perfect Pairing}.
\end{proof}


\begin{lem} \label{Lem: Length of sub-edge}
	Let $G$ be a weighted graph and $G'$ any refinement. If $e' \in E(G')$ is one of the edges in the subdivision 
	of the edge $e \in E(G)$, then 
		\benn
			\int_{e'} = \frac{\ell(e')}{\ell(e)}\int_{r(e)} \ ,
		\eenn
	where $r: C_1(G, \RR) \to C_1(G', \RR)$ is the refinement homomorphism.
\end{lem}

\begin{proof}
Observe that for any harmonic 1-form $\omega = \sum \omega_{\tilde{e}}. d\tilde{e}$ on $G$, we have
 	\[
		\int_{e'} r_*(\omega) = \omega_e \ \ell(e') =  \frac{\ell(e')}{\ell(e)} \omega_e\  \ell(e) 
			=  \frac{\ell(e')}{\ell(e)} \int_e \omega =  \frac{\ell(e')}{\ell(e)} \int_{r(e)} r_*(\omega). 
	\]
The equality of functionals $\int_{e'} = \frac{\ell(e')}{\ell(e)} \int_{r(e)} \in \Omega(G)^*$ follows.
\end{proof}

\begin{proof}[Proof of Theorem~\ref{Thm: Abel-Jacobi Properties}]
If we knew \textit{a priori} that $\Phi_p$ were continuous, then statement (2) would be an immediate consequence of the theory of covering spaces in algebraic topology. Instead, we give a direct construction of the lift $\widetilde{\Phi}_p$ and use it to prove continuity, as well as statements (3) and (4).

Take $\Gamma_1$ to be a simply-connected metric subgraph of 
$\Gamma$ containing the point $p$. Choose a refinement $G$ of $G_0$ such that $V(G)$ contains $p$ and a vertex set for $\Gamma_1$. Let $T$ be the tree in $G$ whose realization is the metric tree $\Gamma_1$. Given a point $q \in \Gamma_1$ and a refinement of $G$ containing $q$ as a vertex, say $G'$, there is a unique minimal path $P$ from $p$ to $q$ in $T'$, the refinement of $T$ in $G'$. Let us view $P$ as a path in $G'$. Define $\widetilde{\Phi}_p(q) = \int_{\alpha_P}$. As the integration pairing is compatible with refinement, the map $\widetilde{\Phi}_p$ is well-defined independent of the choice of $G$ and $G'$. It is clear from the construction that this map is indeed a lift of $\Phi_p$, so (2) is proved. 

Now we give a local formula for the lift $\widetilde{\Phi}_p$. Let $e$ be a segment of $\Gamma$ and suppose $q, q'$ are two points on $e$. Choose a model $G$ containing the vertices $p, q$ and $q'$ such that $q$ and $q'$ are adjacent in $G$. Let $e'$ be a subsegment of $e$ relative to $G$ such that $\iota(e') = (q,q')$, perhaps after swapping $q$ and $q'$. Let $T$ be a tree in $G$ containing the edge $e'$ and the vertex $p$, and write $\Gamma_1$ for the associated metric subtree of $\Gamma$. Define $P$ to be the unique minimal path from $p$ to $q$ along $T$, viewed as a path in $G$, and let $P'$ be the path $P$ followed by the sequence $e', q'$. 

By Lemma~\ref{Lem: Length of sub-edge} we see that
	\[
			\widetilde{\Phi}_p(q') - \widetilde{\Phi}_p(q) = \int_{\alpha_{P'}} - \int_{\alpha_P} = \int_{e'}
			= \frac{\ell(e')}{\ell(e)} \int_{r(e)}.
	\]
Upon noting that $\int_{r(e)}$ and $\int_e$ have the same image in $\Omega(\Gamma)^*$, we deduce that
	\begin{equation} \label{Eq: Straight lines}
		\widetilde{\Phi}_p(q') = 	\widetilde{\Phi}_p(q) + \frac{\ell(e')}{\ell(e)} \int_e.
	\end{equation}

Now we turn to the proofs of (1), (3) and (4). To prove (1), it suffices to show that the lifts $\widetilde{\Phi}_p$ constructed in (2) are continuous. Fix a simply-connected subgraph $\Gamma_1 \subset \Gamma$ containing $p$. Let $q \in \Gamma_1$ and suppose $q'$ is a nearby point on the segment $e$, some part of which is contained in $\Gamma_1$. As $q' \to q$, the length of the segment $e'$ joining them must tend to zero, and the equation \eqref{Eq: Straight lines} shows that $\widetilde{\Phi}_p(q') \to \widetilde{\Phi}_p(q)$.

For (3) and (4), we choose a segment $e$ of $\Gamma$ and suppose that $\Gamma_1$ is a metric tree containing $p$ and $e$. Let $q$ be the tail of $e$, and suppose $q'$ is any other point of $e$. If the complement in $\Gamma$ of the interior of $e$ is disconnected, then $\int_e = 0$ by Lemma~\ref{Lem: Connected vs. cycles}. Formula~\eqref{Eq: Straight lines} shows that $\widetilde{\Phi}_p(q') = \widetilde{\Phi}_p(q)$ for all $q' \in e$, and statement (3) is proved after projecting to $J(\Gamma)$. On the other hand, if the complement of the interior of $e$ is connected, then $\int_e \not= 0$. As the point $q'$ is varied, $\widetilde{\Phi}_p(q')$ traces out a straight segment in $\Omega(\Gamma)^*$ in the direction of $\int_e$. This proves statement (4).

For (5), the result is clear if $\int_e = 0 \in \Omega(\Gamma)^*$, so let us assume otherwise. Then for any
	$\omega = \sum \omega_{\tilde{e}} d\tilde{e} \in \Omega_\ZZ(G)$, we have
		\[
			\frac{1}{\ell(e)}\int_e \omega =  \omega_e \in \ZZ.
		\]
	This proves the first claim. For the second claim, it suffices to produce a 1-form $\omega$ with integral slopes such that $\frac{1}{\ell(e)}\int_e \omega = \pm 1$. By Lemma~\ref{Lem: Connected vs. cycles}, we find that $G-e$ is connected. Choose a spanning tree $T$ that does not contain the edge $e$, and let $\alpha_{T,e}$ be the associated fundamental cycle. Then $\frac{1}{\ell(e)}\int_e d\alpha_{T,e} = 1$.
	
	To prove (6), i.e., show that $\Phi_p$ is tropical, we must verify the three conditions of the definition given at the end of \S\ref{Sec: Tropical Jacobian}. We have already shown above that the Abel-Jacobi map is continuous. If we choose a segment $e$ of $\Gamma$ and a simply-connected subgraph containing $e$ and the point $p$, then we constructed a lift $\widetilde{\Phi}_p: \Gamma_1 \to \Omega(\Gamma)^*$. Equation~\eqref{Eq: Straight lines} implies that $\widetilde{\Phi}_p$ is given locally by 
		\begin{equation} \label{Eq: Tangent computation}
			t \mapsto \widetilde{\Phi}_p(q) + t \left(\frac{1}{\ell(e)}\int_e\right),
		\end{equation}
where $t$ is a unit-speed coordinate on $e$ and $q$ is any point in the interior of $e$. In particular, $\widetilde{\Phi}_p$ carries a primitive integral tangent vector on $\Gamma$ to $\pm \frac{1}{\ell(e)}\int_e \in \Omega_\ZZ(\Gamma)^*$. 

	Finally, we  show that $\Phi_p$ is harmonic. Fix a point $q \in \Gamma$ of valence $v$, and let $\lambda_1, \ldots, \lambda_v$ be the primitive integral tangent vectors at $q$. Equation~\eqref{Eq: Tangent computation} gives a concrete means for describing the image of $\lambda_i$ under the differential map $D\Phi_p$. (To be precise, we should write $(D\Phi_p)_q$, but we drop the subscript $q$ to avoid cluttering notation.) Choose points $q_1, \ldots, q_v$ near $q$ and a model $G$ containing the points $p, q$, and $q_1, \ldots, q_v$ so that
	\[
		D\Phi_p(\lambda_i) = \frac{(-1)^{\epsilon_i}}{\ell(e_i)} \int_{e_i} \ , \qquad i = 1, \ldots, v, 
	\]
where $e_i$ is the segment of $\Gamma$ with endpoints $q$ and $q_i$, and $\epsilon_i = 0$ or~$1$ depending on whether the orientation of $e_i$ is given by $\iota(e_i) = (q, q_i)$ or $\iota(e_i) = (q_i, q)$. Let $\omega= \sum_e \omega_e.de$ be a harmonic 1-form. Then
	\[
		\sum_i \frac{(-1)^{\epsilon_i}}{\ell(e_i)} \int_{e_i} \omega = 
		\sum_i (-1)^{\epsilon_i} \omega_{e_i} = -\sum_{\substack{i \\ e_i^+ = q}} \omega_{e_i}
			+ \sum_{\substack{i \\ e_i^- = q}} \omega_{e_i} = 0.
	\]
Therefore we have the equality of linear functionals $\sum_i D\Phi_p(\lambda_i) = 0$, which completes the proof that $\Phi_p$ is harmonic, and also the proof that $\Phi_p$ is tropical.
\end{proof}


\begin{figure}[ht]
	\begin{center}
	\scalebox{.82}{
	\begin{picture}(100,160)(-45,20)
		\put(-200,60){\includegraphics{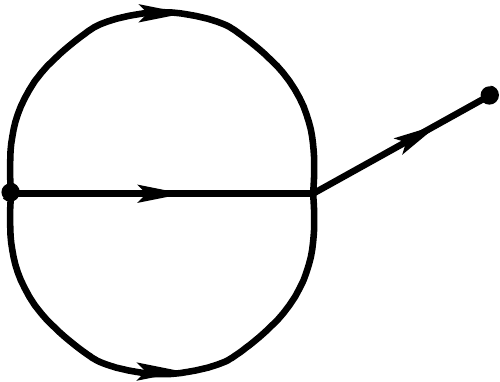}}
		\put(-210, 110){$p$}
		\put(-150, 175){$e_1$}
		\put(-150, 122){$e_2$}
		\put(-150, 71){$e_3$}
		\put(-83, 142){$e_4$}
		\put(-80, 90){$\ell(e_1)=a$}
		\put(-80, 75){$\ell(e_2)=b$}
		\put(-80, 60){$\ell(e_3)=c$}
		\put(-200, 50){$\Gamma$}
		
		\put(80,30){\includegraphics{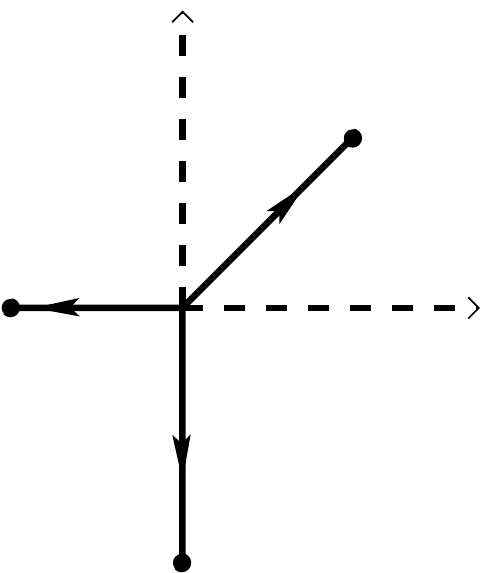}}
		\put(185, 164){$(a,a)$}
		\put(45, 95){$(-b,0)$}
		\put(140, 28){$(0,-c)$}
		\put(165,125){$\Phi_p(e_1)$}
		\put(85,115){$\Phi_p(e_2)$}
		\put(140,70){$\Phi_p(e_3)$}
		\put(50, 50){$\Omega(\Gamma)^*$}
	\end{picture}
	}
	\end{center}
	
	\caption{
		At left we have a diagram of a tropical curve $\Gamma$, implicitly given by a model $G$ with
		four edges $\{e_1, e_2, e_3, e_4\}$ and three vertices. At right, the bold segments form the image of 
		the tropical Abel-Jacobi map $\Phi_p$ inside $\Omega(\Gamma)^*$. 
		The dashed coordinate axes are drawn relative to the ordered integral 
		basis $\{\int_A, \int_B\}$ for $\Omega_\ZZ(\Gamma)^*$. See Example~\ref{Example: Theta}
		for further discussion. See also \cite[Fig.~7, \S6.2]{Mikhalkin_Zharkov_Tropical_Jacobians_2008} 
		for another illustration of this example. 
		}
	\label{Fig: Theta}
\end{figure}
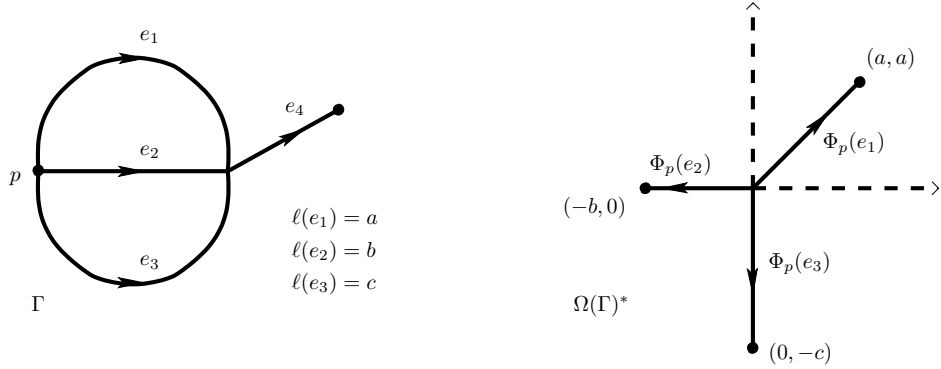

\begin{figure}[ht]
	\scalebox{.69}{\includegraphics{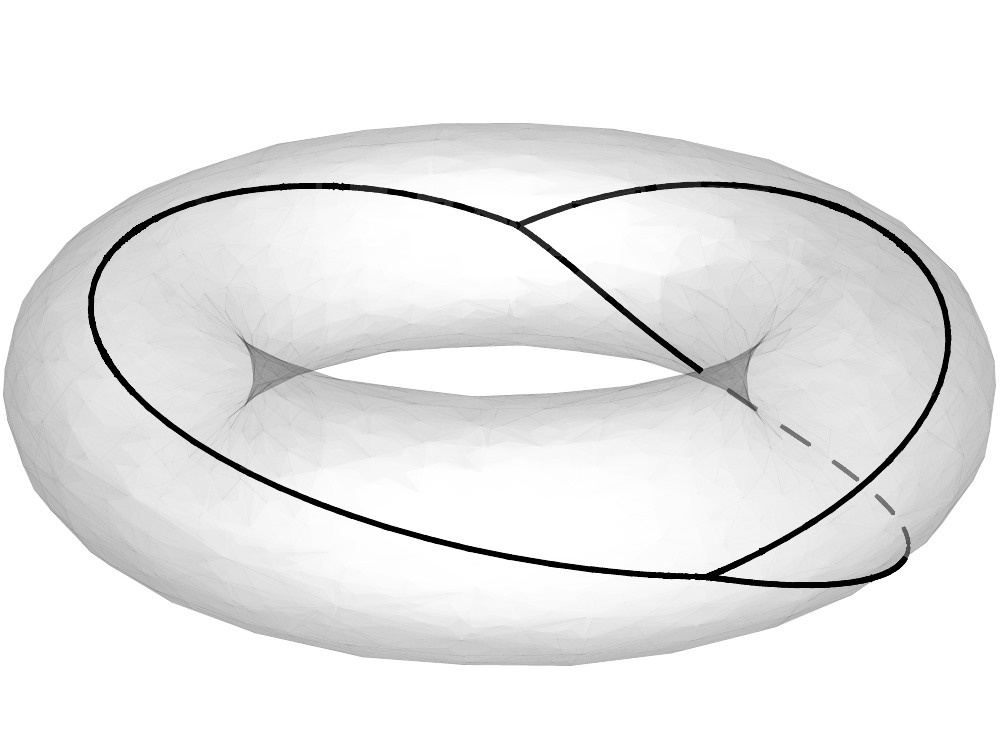}}

	\caption{
		Let $\Gamma$ be the tropical curve depicted in Figure~\ref{Fig: Theta}. Here we see an illustration of
		the Jacobian $J(\Gamma)$ embedded as a hollow torus in $\RR^3$. The black curve is the image of 
		$\Gamma$ via the Abel-Jacobi map. (Figure drawn with \textit{Mathematica} 6.0.1.0.)
		}
	\label{Fig: Torus Theta}
\end{figure}

\begin{example} \label{Example: Theta}
	Consider the tropical curve $\Gamma$ illustrated in Figure~\ref{Fig: Theta}. For the indicated 
	basepoint $p$, we want to compute the image of the Abel-Jacobi map $\Phi_p$. 
	
	In order to describe the basic functionals $\int_e$ in some meaningful way, we choose a basis for 
	$\Omega_\ZZ(\Gamma)^*$. 
	Let $d\alpha = de_1 - de_2$ and $d\beta = de_1 - de_3$ be a basis of $\Omega(\Gamma)$, and by duality write 
	$A, B \in H_1(\Gamma, \RR)$ for the cycles such that
		\[
			\int_A d\alpha = \int_B d\beta = 1, \qquad \int_A d\beta = \int_B d\alpha = 0.
		\]
	Then $\{\int_A, \int_B\}$ generate $\Omega_\ZZ(\Gamma)^*$ as a $\ZZ$-module. It follows that
		\begin{equation} \label{Eq: Tropical Theta}
			\int_{e_1} = a \int_A + a \int_B \ , \qquad
			\int_{e_2} = -b \int_A \ , \qquad \int_{e_3} = -c \int_B \ , \qquad  \int_{e_4} = 0,
		\end{equation}
	as can easily be verified by evaluation on the basis $\{d\alpha, d\beta\}$ 
	for $\Omega(\Gamma)$. 
	
	If $\vec{v}$ is a primitive integral tangent vector at $p$ pointing along the edge $e$, then we can speak of the 
	point $p + t\vec{v}$ at distance $t$ from $p$, at least if $t \leq \ell(e)$. 
	Lemma~\ref{Lem: Length of sub-edge} shows that 
	$\widetilde{\Phi}_p(p + t\vec{v}) = \frac{t}{\ell(e)}\int_e$ in 
	$\Omega(\Gamma)^*$. So when $e = e_1$, for example, 
	we find that
		\[
			\widetilde{\Phi}_p(p + t\vec{v}) = t\left( \int_A   +\int_B\right) , \quad 0 \leq t \leq a.
		\]
	Plotting this in coordinates relative to the ordered basis $\{\int_A, \int_B\}$ gives the illustration on the right
	in Figure~\ref{Fig: Theta}. 
	
	The lattice $H_1(\Gamma, \ZZ)$ is generated by $e_1 - e_2$ and $e_1 - e_3$, and so its image
	in $\Omega(\Gamma)^*$ is generated by 
		\[
			\int_{e_1 - e_2} = (a+b)\int_A +\  a\int_B \ , \qquad \int_{e_1-e_3} = a \int_A +\  (a+c)\int_B;
		\]
	i.e., generators are given by $(a+b, a)$ and $(a, a+c)$ using coordinates relative to $\{\int_A, \int_B\}$. 
	In particular, this shows that the endpoints $(a,a), (-b,0), (0,-c)$ of the three segments in $\widetilde{\Phi}_p(\Gamma)$ agree 
	modulo the lattice $H_1(\Gamma, \ZZ)$. 
	An illustration of $\Phi_p(\Gamma)$ is given in 
	Figure~\ref{Fig: Torus Theta}
\end{example}


\section{Metric Structures on the Tropical Jacobian}
\label{Sec: Metrics}

	Suppose we are given a tropical curve $\Gamma$, a tropical torus $(V/\Lambda, \Lambda')$, and a tropical map $f: \Gamma \to V/\Lambda$. If we further endow $V/\Lambda$ with some kind of metric data, then one can ask how the length of a segment $e \subset \Gamma$ compares with the length of its image $f(e)$.  Near a point $q \in \Gamma$ there is a neighborhood $U_q$ and a continuous lift $\tilde{f}: U_q \to V$ such that segments of $U_q$ map to straight segments in $V$. Each such segment in $V$ is a translate of a scalar multiple of some lattice vector in $\Lambda'$, so it suffices to define the lengths of elements of $\Lambda'$ in order to prescribe a well-defined length to the image segment $f(e)$.
	
	In this section, we will give three canonical notions of length for the lattice $\Omega_\ZZ(\Gamma)^*$ defining the tropical structure on the Jacobian $J(\Gamma)= \Omega(\Gamma)^* / H_1(\Gamma, \ZZ)$. Each has its own merits, but the common theme is that we can give explicit formulas for each in terms of either combinatorial graph theory or circuit theory.
	
\subsection{The Tropical Metric}
	
\begin{define}
	Let $(V/\Lambda, \Lambda')$ be a tropical torus, let 
	$\RR \Lambda' = \{ t \lambda' \; : \; t \in \RR, \lambda' \in \Lambda' \}$ be the union of all lines in $V$ containing some nonzero $\lambda' \in \Lambda'$, and let $v \in \RR\Lambda'$ be any vector.  
	Write $v = t\lambda$ with $t \in \RR$ and $\lambda \in \Lambda'$ primitive.
	The \textbf{tropical length} of $v=t\lambda$ is defined to be 
	$\|v\|_\ZZ = |t|$. 
	We call $\|\cdot\|_\ZZ$ the \textbf{tropical metric} on $\RR\Lambda'$.
\end{define}

The function $\| \cdot \|_\ZZ : \RR\Lambda' \to \RR_{\geq 0}$ is \textit{not} a norm in the usual sense (unless $V$ has
dimension~$1$). For example, let
$V = \RR^2$ and $\Lambda' = \ZZ^2$. Then 
$\| (2,1) \|_\ZZ = \| (1,2) \|_\ZZ = 1$, but
$\left\|(2,1) + (1,2) \right\|_\ZZ = 3$. Hence the triangle inequality fails for $\|\cdot\|_\ZZ$. This will cause no problems for us, though, because we will only be using the tropical metric to define lengths of straight segments in $V$ (which are one-dimensional objects).

\begin{thm} \label{Thm: Tropical Isometry}
	Let $\Gamma$ be a tropical curve, and let $p \in \Gamma$ be a basepoint. Then the Abel-Jacobi map 
	$\Phi_p: \Gamma \to J(\Gamma)$ is a tropical isometry along any segment of $\Gamma$ that is not collapsed to a point.
\end{thm}

\begin{proof}

	We already know the Abel-Jacobi map is tropical. To check that it is an isometry, we work locally. Let $q \in \Gamma$ be a point of valence~2 on some segment~$e$, and let $\Gamma_1  \subset \Gamma$ be a simply-connected metric subgraph containing the basepoint~$p$ and the segment~$e$. By Theorem~\ref{Thm: Abel-Jacobi Properties}(2),  there is a lift $\widetilde{\Phi}_p: \Gamma_1 \to \Omega(\Gamma)^*$, and we saw in \eqref{Eq: Tangent computation} that it can be computed locally by 
	\[
		t \mapsto \widetilde{\Phi}_p(q) + t \left( \frac{1}{\ell(e)} \int_e \right).
	\]
If $D\Phi_p$ is the induced map on tangent spaces and $\vec{v} = \frac{d}{dt}$ is a primitive integral tangent vector 
at $q$, then
	\[
		\left\|D\Phi_p(\vec{v})\right\|_\ZZ = \left\| \frac{1}{\ell(e)} \int_e \right\|_\ZZ.
	\]
	
	Suppose now that $e$ is not collapsed by the Abel-Jacobi map. Then  $\int_e \not= 0$ in $\Omega(\Gamma)^*$, and Theorem~\ref{Thm: Abel-Jacobi Properties}(5) shows that $\frac{1}{\ell(e)}\int_e$ is a primitive element of  the lattice $\Omega_\ZZ(\Gamma)^*$. In particular, it has tropical length~1, and so the above computation proves that $\Phi$ carries unit-length tangent vectors to unit-length tangent vectors. Hence the lengths of $e$ and of the image segment $\Phi_p(e)$ agree.
\end{proof}


\subsection{The Foster/Zhang Metric}
\label{Section: Foster/Zhang Metric}

	The Foster/Zhang metric is most easily defined in terms of a model for $\Gamma$, although it does not depend on 
	the choice of model. Let $G$ be a weighted graph and enumerate the edges of $G$ by $e_1, \ldots, e_m$. For 
	any 1-chain $\alpha = \sum_j a_j . e_j \in C_1(G, \RR)$, we define its \textbf{supremum norm} to be 
	\benn
		\|\alpha\|_{\sup} = \max_{1 \leq j \leq m} |a_j|.
	\eenn

		Let $\pi: C_1(G, \RR) \to H_1(G, \RR) \subset C_1(G, \RR)$
		be the orthogonal projection map relative to the inner product defined in \S\ref{Sec: Wtd graphs}. 
		The \textbf{supremum norm} of an element $\int_\alpha \in \Omega(G)^*$ is defined to be
		\[
			\left\|\int_\alpha\right\|_{\sup} = \|\pi(\alpha)\|_{\sup},
		\]
	where $\alpha \in C_1(G, \RR)$. This quantity is well-defined by Lemma~\ref{Lem: Perfect Pairing}; i.e., 
	if $\int_\alpha = \int_\beta$, then  $\alpha = \beta +df  \Rightarrow \pi(\alpha) = \pi(\beta)$.
	
\begin{lem} \label{Lem: Foster refinement}
	Let $G$ and $G'$ be weighted graphs such that $G'$ refines $G$, and let $r: C_1(G, \RR) \to C_1(G',\RR)$ 
	be the refinement homomorphism on chain spaces. Then the induced isomorphism 
	$\Omega_1(G, \RR)^* \to \Omega_1(G', \RR)^*$ given by $\int_\alpha \mapsto \int_{r(\alpha)}$ is an 
	isometry for the supremum norm.
\end{lem}

\begin{proof}
	Suppose $\alpha \in H_1(G, \RR)$. Then $\pi(\alpha) = \alpha$ and $\pi'(r(\alpha)) = r(\alpha)$, where
	$\pi':C_1(G', \RR) \to H_1(G', \RR)$ is the orthogonal projection for $G'$. So it suffices to prove that
	$\|\alpha\|_{\sup} = \|r(\alpha)\|_{\sup}$. But if we expand $\alpha$ and $r(\alpha)$ in terms of the 
	edge bases for $C_1(G, \RR)$ and $C_1(G',\RR)$, they have the same set of coefficients, and hence the same supremum norm.
\end{proof}

\begin{define}
	An element $s \in \Omega(\Gamma)^*$ can be represented by choosing a model $G$ and a functional 
	$\int_\alpha \in \Omega(G)^*$. Define 
	$\|s\|_{\sup} = \left\|\int_\alpha \right\|_{\sup}$. By the previous lemma, this quantity is well-defined
	independent of the choice of model. We call $\|\cdot\|_{\sup}$ the \textbf{Foster/Zhang metric} on 
	$\Omega(\Gamma)^*$. 
	\end{define}

	A particular case of the above constructions deserves a special name:
	
\begin{define}
	Let $G$ be a weighted graph. The \textbf{Foster coefficient} associated to an edge $e$ of $G$ is
	\benn
		F(e) = \|\pi(e)\|_{\sup}.
	\eenn
\end{define}

We will now show that the Foster coefficients arise naturally as the lengths (with respect to the Foster/Zhang metric) of images of segments under the Abel-Jacobi map. In the next section, we give  interpretations of the Foster coefficients in terms of spanning trees, circuit theory, and random walks.

In the statement of the next result, we let $g = g(\Gamma)$ denote the
 \textbf{genus} of $\Gamma$, i.e.,
$g(\Gamma) = \operatorname{rk} H_1(\Gamma, \ZZ) = \dim J(\Gamma)$.
For any model $G$ of $\Gamma$, we have $g(\Gamma) = \#E(G)-\#V(G)+1$.
			
\begin{thm} \label{Thm: Foster Metric}
	Let $\Gamma$ be a tropical curve, and let $p \in \Gamma$ be a basepoint. If we endow the Jacobian 
	$J(\Gamma)$ with the Foster/Zhang metric, then the image of a segment $e \subset \Gamma$ 
	under the Abel-Jacobi map $\Phi_p: \Gamma \to J(\Gamma)$ has length $F(e)$. The image graph $\Phi_p(\Gamma)$ has total length $g(\Gamma)$, and each segment
of the image graph has length at most 1, with equality if and only if the segment forms a loop.
\end{thm}

\begin{proof}
	We use exactly the same setup as in the proof of Theorem~\ref{Thm: Tropical Isometry}, so that
		\[
			\|D\Phi_p(\vec{v})\|_{\sup}= \left\|\frac{1}{\ell(e)}\int_e \right\|_{\sup} 
				= \frac{\|\pi(e)\|_{\sup}}{\ell(e)} = \frac{F(e)}{\ell(e)}.
		\]
	Hence the local scaling factor for the Foster/Zhang metric is $F(e) / \ell(e)$, which means the segment $e$ 
	maps to a segment of length $F(e)$ under the Abel-Jacobi map. 
	
	The claim about the total length of the image graph is an immediate  
	consequence of Foster's Theorem, given below as Corollary~\ref{Thm: Foster}. Finally, to see that a segment
	has length at most 1, we observe that Theorem~\ref{Thm: Foster Interpretation} below gives
	$F(e) = 1 - r(e) / \ell(e)$, where $r(e)$ is the effective resistance between the endpoints of the segment $e$. 
	The inequality $r(e) \leq \ell(e)$ always holds \cite[\S6]{Baker_Faber_2006}, and $r(e) = 0$ if and only if
	$e^- = e^+$, that is, $e$ is a loop in $\Gamma$.
\end{proof}

	\begin{remark} \label{FosterSubdivisionRemark}
	As a consequence of Theorem~\ref{Thm: Foster Metric} and the fact that segments of $\Gamma$ map to 
	straight segments in $\Omega(\Gamma)^*$, we see that the Foster coefficients
	are additive with respect to edge refinements. That is, if $G'$ is a refinement of $G$, and if the edge $e$ of $G$
	 is subdivided into $e_1, \ldots, e_s$ in $G'$, then
		\[
			F(e) = F(e_1) + \cdots + F(e_s).
		\]
	In fact, Lemma~\ref{Lem: Length of sub-edge} gives the stronger formula
	$F(e_i) = \frac{\ell(e_i)}{\ell(e)}F(e)$ for $i = 1, \ldots, s$.
	\end{remark}

\subsection{The Euclidean Metric}

The Euclidean metric is defined in a similar fashion to that of the Foster/Zhang metric. We begin by letting $G$ be a weighted graph and $\pi: C_1(G, \RR) \to H_1(G, \RR)$ the orthogonal projection relative to the inner product $\ip{\cdot}{\cdot}$ defined in \S\ref{Sec: Wtd graphs}. For any $1$-chain $\alpha \in C_1(G, \RR)$, we define its \textbf{Euclidean norm} to be 
	\[
		\|\alpha\|_{\Euclidean} = \sqrt{\ip{\alpha}{\alpha}},
	\]
and we define the \textbf{Euclidean norm} of an element $\int_\alpha \in \Omega(G)^*$ by
	\[
		\left\|\int_\alpha\right\|_{\Euclidean} = \|\pi(\alpha)\|_{\Euclidean}.
	\]
The Euclidean norm is well-defined by Lemma~\ref{Lem: Perfect Pairing}.

\begin{lem} \label{Lem: L^2 refinement}
	Let $G$ and $G'$ be weighted graphs such that $G'$ refines $G$, and let $r: C_1(G, \RR) \to C_1(G',\RR)$ 
	be the refinement homomorphism on chain spaces. Then the induced isomorphism 
	$\Omega_1(G, \RR)^* \to \Omega_1(G', \RR)^*$ given by $\int_\alpha \mapsto \int_{r(\alpha)}$ is an 
	isometry for the Euclidean norm.
\end{lem}

\begin{proof}
	Suppose $\alpha \in H_1(G, \RR)$. Then $\pi(\alpha) = \alpha$ and $\pi'(r(\alpha)) = r(\alpha)$, where
	$\pi':C_1(G', \RR) \to H_1(G', \RR)$ is the orthogonal projection for $G'$. So it suffices to prove that
	$\ip{\alpha}{\alpha} = \ip{r(\alpha)}{r(\alpha)}$. By linearity of the inner product and the fact that distinct
	edges are orthogonal, we are reduced to proving $\ip{e}{e} = \ip{r(e)}{r(e)}$ for an arbitrary edge $e$ of $G$, 
	in which case the result is immediate from the definitions.
\end{proof}

\begin{define}
	An element $s \in \Omega(\Gamma)^*$ can be represented by choosing a model $G$ and a functional 
	$\int_\alpha \in \Omega(G)^*$. Define 
	$\|s\|_{\Euclidean} = \left\|\int_\alpha \right\|_{\Euclidean}$. By the previous lemma, this quantity is 
	independent of the choice of model. We call $\|\cdot\|_{\Euclidean}$ the \textbf{Euclidean metric} on 
	$\Omega(\Gamma)^*$. 
\end{define}

\begin{thm} \label{Thm: L^2 Metric}
	Let $\Gamma$ be a tropical curve, and let $p \in \Gamma$ be a basepoint. If we endow the Jacobian 
	$J(\Gamma)$ with the Euclidean metric, then the image of a segment $e \subset \Gamma$ 
	under the Abel-Jacobi map $\Phi_p: \Gamma \to J(\Gamma)$ has length $\sqrt{\ell(e)F(e)}$, where
	$F(e)$ is the Foster coefficient defined in the previous section. 
\end{thm}

\begin{proof}
	Again we use the setup from the proof of Theorem~\ref{Thm: Tropical Isometry}, 
	so that
		\[
			\left\|D\Phi_p(\vec{v}) \right\|_{\Euclidean}
			= \left\| \frac{1}{\ell(e)}\int_e \right\|_{\Euclidean} =\frac{1}{\ell(e)} \sqrt{\ip{\pi(e)}{\pi(e)}}.
		\]
	
	Enumerate the edges of $G$ as $e_1, \ldots, e_m$ in such a way that $e_1 = e$, and
	write $\pi(e_1) = \sum_i a_{i1} e_i$ for some real numbers $a_{i1} \in \RR$. Since $\pi$ is the orthogonal
	projection, we see $\ip{\pi(\alpha)}{\pi(\beta)} = \ip{\pi(\alpha)}{\beta}$ for any $1$-chains $\alpha, \beta$. 
	In particular, for $\alpha = \beta = e_1$, we get
		\[
			\ip{\pi(e_1)}{\pi(e_1)} = \ip{\pi(e_1)}{e_1} = 
			a_{11}\ell(e_1).
		\]
	By Corollary~\ref{Cor: Foster-Trees} below, we have $a_{11} = F(e_1)$. Replacing $e_1$ by $e$, 
	we have proved that $\ip{\pi(e)}{\pi(e)} = \ell(e)F(e)$, which implies that the local scaling factor for the Abel-Jacobi
	map is $\sqrt{F(e) / \ell(e)}$. In particular, a segment of length $\ell(e)$ maps to a segment of length
		$\ell(e) \sqrt{F(e) / \ell(e)} = \sqrt{\ell(e)F(e)}$.
\end{proof}

\begin{example}
	Continuing Example~\ref{Example: Theta}, we want to compute the
	length of $\Phi_p(e_i)$ relative to the various metrics we've defined. The tropical length can be calculated
	immediately because we have written $\int_{e_i}$ in terms of the basis $\{\int_A, \int_B\}$ for 
	$\Omega_\ZZ(\Gamma)^*$. Indeed, looking at~\eqref{Eq: Tropical Theta} and using the fact that
	 $\int_A, \int_B $ and $\int_A + \int_B$ are all primitive integral vectors, we find
		\[
			\|\Phi_p(e_1)\|_\ZZ = a, \qquad \|\Phi_p(e_2)\|_\ZZ = b, \qquad \|\Phi_p(e_3)\|_\ZZ = c, \qquad
			\|\Phi_p(e_4)\|_\ZZ = 0.
		\]
	As $e_1, e_2, e_3$ are each part of some cycle, the Abel-Jacobi map is an isometry on these edges for the 
	tropical metric. This is in perfect agreement with Theorem~\ref{Thm: Tropical Isometry}.

The associated Foster coefficients can be computed either directly by orthogonal projection, by averaging over spanning trees 
	as in Corollary~\ref{Cor: Foster-Trees}, or via circuit theory as in
	Theorem~\ref{Thm: Foster Interpretation}. The result is that
		\[
			F(e_1) = \frac{a(b+c)}{ab+ac+bc}, \qquad F(e_2) = \frac{b(a+c)}{ab+ac+bc}, 
		\]
		\[
			F(e_3) = \frac{c(a+b)}{ab+ac+bc}, \qquad F(e_4) = 0.
		\]
	Theorem~\ref{Thm: Foster Metric} implies that 
	 $\|\Phi_p(e_i)\|_{\sup} = F(e_i)$, while Theorem~\ref{Thm: L^2 Metric} shows that $\|\Phi_p(e_i)\|_{\Euclidean} = \sqrt{\ell(e_i)F(e_i)}$.
\end{example}


\section{Interpretations of the Foster Coefficients}
\label{Sec: Foster Coefficients}

As mentioned in the introduction, the present work was originally motivated by a desire to realize the canonical measure $\mu_{\Zh}$ which appears in Zhang's paper \cite{Zhang_1993}
as ``coming from'' the tropical Abel-Jacobi map.  In order to make the connection between the Foster coefficients defined in
\S\ref{Section: Foster/Zhang Metric} and Zhang's explicit formula for $\mu_{\Zh}$, we first relate the Foster coefficients to spanning trees and electrical networks. We also mention a connection with random walks on graphs.  

With the exception of our tropical interpretation of Zhang's measure in \S\ref{Sec: Zhang Measure}, much of the material in this section is likely to be well-known to the experts in graph theory or Arakelov theory. We claim originality only in our focus on the Foster coefficients as the central quantities of study. 

\subsection{Spanning Trees}

	In this section we give an interpretation of the Foster coefficients in terms of weighted spanning trees. Let $G$ be a weighted graph. Define the \textbf{weight} of a spanning tree $T$ in $G$ to be 
	\benn
		w(T) = \prod_{e \not\in E(T)} \ell(e),
	\eenn
the product of the lengths of all edges of $G$ not in the spanning tree. Define $w(G)$ to be the sum of all of the weights $w(T)$ over all spanning trees $T$. In particular, if $G$ is itself a tree, then $w(G)=1$. Recall that a spanning tree $T$ provides a basis of $H_1(G, \RR)$ consisting of fundamental cycles $\{\alpha_{T, e} : e \in E(G) \smallsetminus E(T)\}$. By convention, we set $\alpha_{T,e}=0$ if $e \in E(T)$.  The following fact is due to Kirchhoff. 

\begin{thm} \label{Thm: Projections-Trees}
	Let $G$ be a weighted graph with edges $e_1, \ldots, e_m$. The orthogonal projection $\pi: C_1(G, \RR) \to H_1(G,\RR)$ is given by $\pi(e_i) = \sum a_{ij} e_j$, where the coefficients $a_{ij}$ may be computed via the formula
		\benn
			a_{ij} = \frac{1}{w(G)} \sum_T 
				\alpha_{T, e_i}(e_j) \ w(T),
		\eenn
	the sum being over all spanning trees of $G$.
\end{thm}

\begin{proof}
For each spanning tree $T$, define a linear operator $M_T: C_1(G, \RR) \to C_1(G, \RR)$ by its values on the canonical edge basis of $C_1(G, \RR)$:
	\benn
			M_T(e) = \alpha_{T, e},
	\eenn
where we use the convention that $\alpha_{T, e} = 0$ if $e$ is an edge of $T$. Note that $\im(M_T) = H_1(G, \RR)$. Define $M: C_1(G, \RR) \to C_1(G, \RR)$ to be an average of these operators:
	\benn
		M = \frac{1}{w(G)} \sum_T w(T) M_T.
	\eenn
It is proven in \cite[Prop.~7.2~\&~15.1]{Biggs_Potential_Theory_1997} that $M$ is the orthogonal projection matrix, from which the theorem easily follows. 
\end{proof}

	Now we show the influence of the theorem on the computation of Foster coefficients. The next corollary shows that in the matrix representation for $\pi$ with respect to the canonical basis of $C_1(G, \RR)$, the Foster coefficients are precisely the entries on the diagonal. Flanders \cite{Flanders_Foster_1974} seems to have been the first to notice that the trace of the projection matrix can be interpreted in an interesting way (see~Theorem~\ref{Thm: Foster}). 
	
\begin{cor} \label{Cor: Foster-Trees}
	For a weighted graph $G$ and an edge $e_j$, set $\pi(e_i) = \sum a_{ij} e_j$ as in the theorem.  
	The Foster coefficient $F(e_i)$ is given by
		\benn
			F(e_i) = a_{ii} = \frac{1}{w(G)} \sum_
			{\substack{T \\ e_i \not\in E(T)} } w(T).
		\eenn
\end{cor}

\begin{proof}[Proof of Corollary]
	Considering the definition of $F(e_i)$,
it suffices to show that $|a_{ii}| \geq |a_{ij}|$ for all $j$,
which follows immediately from the theorem as $|\alpha_{T, e_i}(e_j)| \leq 1$ for all $j$, and $\alpha_{T, e_i}(e_i) = 1$ for every tree $T$ not containing the edge~$e_i$. 
\end{proof}

\begin{remark}
	Even though the maximum defining $F(e_i)$ is attained on the diagonal coefficient, it is not necessarily uniquely attained there. 
\end{remark}

\subsection{Circuit Theory and Random Walks}
\label{Sec: Circuits}

	Given a weighted graph $G$, we can view it as an electrical network by identifying each vertex of $G$ with a node, and each edge $e$ of $G$ with a wire (branch) of resistance $\ell(e)$. In this setting, we can talk about the \textbf{effective resistance} between two nodes $x,y$, which will be denoted by $r(x,y)$. For example, it is symmetric in the nodes $x$ and $y$, and it obeys the classical parallel and series rules of circuit reduction. For an exposition of the mathematical theory of effective resistance that is well-adapted to our needs, see \cite[\S6]{Baker_Faber_2006}. (See also \cite{Coyle-Lawler_1999, Doyle_Snell_1984}.)	
	
	For an edge $e$, let us agree to write $r(e) = r(e^-, e^+)$ for the effective resistance between the endpoints of $e$. Let us also write $\mathcal{R}(e)$ for the effective resistance between $e^-$ and $e^+$ on the graph $G - e$ obtained by deleting the edge $e$. If $G - e$ is disconnected, we set $\mathcal{R}(e) = \infty$. 
	
\begin{thm} \label{Thm: Foster Interpretation}
	Let $G$ be a weighted graph. For each edge $e \in E(G)$, the Foster coefficient $F(e)$ satisfies the identities
		\[
			F(e) = 1 - \frac{r(e)}{\ell(e)} = \frac{\ell(e)}{\ell(e) + \mathcal{R}(e)},
		\]
	where the final quantity is interpreted as $0$ if $\mathcal{R}(e) = \infty$.
\end{thm}
	
	Before giving the proof, we state a theorem of R.M. Foster, proved in 1949, that
	relates the metric data of $F(e)$ to the combinatorial data of a weighted graph $G$. It 
	explains our motivation in coining the term \textit{Foster coefficient}.

\begin{cor}[Foster's Theorem, \cite{Foster_Theorem_1949}] \label{Thm: Foster}
	Let $G$ be a weighted graph, and let $g = \#E(G)-\#V(G)+1$. Then
		\benn
			\sum_{e \in E(G)} F(e) = \sum_{e \in E(G)} \frac{\ell(e)}{\ell(e) + \mathcal{R}(e)}  = g.
		\eenn
\end{cor}

\begin{proof}[Proof of the Corollary]
	The first equality is immediate from Theorem~\ref{Thm: Foster Interpretation}. 
	To see the second equality, note that Corollary~\ref{Cor: Foster-Trees} gives
	\benn
		w(G) \sum_e F(e) = \sum_e \sum_{\substack{T \\ e \not\in E(T)}} w(T)
		= \sum_T w(T) \sum_{e \not\in E(T)} 1 =  w(G)g,
	\eenn
since there are precisely $g$ edges in the complement of any spanning tree. 
\end{proof}

\begin{proof}[Proof of Theorem~\ref{Thm: Foster Interpretation}]
	There are two cases to consider for the first equality.  We start by assuming that~$e$ is a loop. Then no spanning tree contains $e$, and hence Corollary~\ref{Cor: Foster-Trees} gives
	\[
		\sum_{\substack{T \\ e \not\in E(T)}} w(T) = w(G) \Rightarrow F(e) = 1.
	\]
Also, $e^+ = e^-$, so $r(e) = 0$. Therefore $1 - r(e) / \ell(e) = 1$.

Now suppose $e$ is not a loop. Then
	\[
		w(G)F(e) = \sum_{\substack{T \\ e \not\in E(T)}} w(T) = w(G)
			- \sum_{\substack{T \\ e \in E(T)}} w(T).
	\]
Observe that there is a bijective correspondence between spanning trees $T$ in $G$ containing the edge $e$ and spanning trees $T/e$ in $G/e$, the graph obtained by contracting the edge $e$. Moreover, $w(T) = w(T/e)$ under this correspondence. Hence $w(G)F(e) = w(G) - w(G/e)$. By the proof of \cite[Lem.~4.5]{Faber_EGBC_2009}\footnote{In \cite{Faber_EGBC_2009}, the notation $\eta(\ell_1, \ldots, \ell_m)$ is used for $w(G)$,
and $R_k(\ell_1, \ldots, \ell_m)$ is used for $r(e_k) w(G)$, 
where $\ell_1, \ldots, \ell_m$ are the lengths of the edges $e_1, \ldots, e_m$.
},
we get $w(G/e) = r(e)w(G)/\ell(e)$. Inserting this into the previous equality and dividing by $w(G)$ gives the result.

	Next observe that we can view $G$ as a circuit in parallel between the two vertices $e^-$ and $e^+$; the edge $e$ and the complement of the edge $e$ form the two parallel pieces. Then the parallel rule for circuit reduction \cite[Thm.~9(iii)]{Baker_Faber_2006} shows that 
	\[
		r(e) = r(e^-,e^+) = \left( \frac{1}{\ell(e)} + \frac{1}{\mathcal{R}(e)}\right)^{-1} = 
			\frac{\ell(e)\mathcal{R}(e)}{\ell(e) + \mathcal{R}(e)}.
	\]  
We interpret the final quantity to be $\ell(e)$ if $\mathcal{R}(e) = \infty$.
The result follows upon inserting this calculation into the formula $F(e) = 1 - r(e) / \ell(e)$ and simplifying. 
\end{proof}

\begin{remark}
There is a well-known connection between electrical networks and random walks on graphs; see \cite{Doyle_Snell_1984} for a beautiful introduction to this subject.
We briefly indicate (without proof) a random walk interpretation of the Foster coefficients.
If $G$ is a weighted graph, define a random walk on $V(G)$ as follows: if $v \in V(G)$ and $e \in E(G)$ is an edge having $v$ as 
one endpoint and $w$ as the other, then a random walker starting at $v$ will move to $w$ along $e$ with probability 
\[
\frac{1/\ell(e)}{\sum_{e' \in E(G), v \in e'} 1/\ell(e')}.
\]

Let $P(e)$ be the probability that a random walker
starting at $v$ reaches $w$ without ever traversing the edge $e$.
Then $P(e) = F(e)$. 

\end{remark}


\subsection{Zhang's canonical measure}
\label{Sec: Zhang Measure}

If $X$ is a compact Riemann surface of genus $g \geq 1$, then 
there is a well-known canonical $(1,1)$-form $\omega_{X}$ on $X$,
known variously as the \textbf{canonical $(1,1)$-form}, the \textbf{Arakelov $(1,1)$-form}, or (by abuse of terminology) the \textbf{Bergman metric}.    This $(1,1)$-form can be described explicitly as follows. Let $\omega_1,\ldots,\omega_g$ be an orthonormal basis for the space $H^0(X,\Omega^1_X)$ of holomorphic $1$-forms on $X$ with respect to the Hermitian inner product $(\omega,\eta) = \frac{i}{2}
\int_X \omega \wedge \overline{\eta}$.  Then
\begin{equation}
\tag{RS1}
\label{eq: OmegaXDef}
\omega_{X} = \frac{i}{2} \sum_{k=1}^g \omega_k \wedge \overline{\omega}_k.
\end{equation}
This description is independent of the choice of orthonormal basis for $H^0(X,\Omega^1_X)$. 

Integrating $\omega_{X}$ over $X$ shows that its volume is equal to the genus of $X$:
\begin{equation}
\tag{RS2}
\label{eq: VolX}
\int_X \omega_{X} = g.
\end{equation}


Let $J(X)$ be the Jacobian of $X$, and let $\Phi_p : X \to J(X)$ be
the Abel-Jacobi embedding relative to a fixed base point $p \in X$.
There is a canonical translation-invariant $(1,1)$-form $\omega_J$
on $J(X)$, coming from the identification of $J(X)$ with the complex torus $H^0(X,\Omega^1_X)^*/H_1(X,\ZZ)$, whose associated K{\"a}hler metric is induced by the flat Hermitian metric on 
$H^0(X,\Omega^1_X)$ (i.e., the inner product on $H^0(X, \Omega^1_X)$ defined above).  
A basic fact is that $\omega_{X}$ is obtained by pulling back $\omega_J$ along the Abel-Jacobi map:
\begin{equation}
\tag{RS3}
\label{eq: OmegaPullback}
\omega_{X} = \Phi_p^*(\omega_J).
\end{equation}


\medskip

We seek tropical analogues of (\ref{eq: OmegaXDef}),
(\ref{eq: VolX}), and (\ref{eq: OmegaPullback}) with $X$ replaced
by a tropical curve $\Gamma$ and $J(X)$ replaced by $J(\Gamma)$.
The role of the canonical $(1,1)$-form $\omega_X$ will be played by
the canonical measure $\mu_{\Zh}$ defined by S. Zhang in 
\cite{Zhang_1993}.

Let $\Gamma$ be a tropical curve of genus $g \geq 1$, and fix
a model $G$ for $\Gamma$.
The canonical measure on $\Gamma$ is defined by the formula
\[
\mu_{\Zh} = \sum_{e \in E(G)} F(e) \frac{dx|_e}{\ell(e)}
= \sum_{e \in E(G)} \frac{dx|_e}{\ell(e) + \mathcal{R}(e)},
\]
where $dx|_e$ is the Lebesgue measure on $e \cong [0,\ell(e)]$
and $F(e)$ 
is the corresponding Foster coefficient. 
If we refine $G$ by subdividing an edge $e$ into two edges $e'$ and $e''$, and if $f: e \to \RR$ is a continuous function, then by
Remark~\ref{FosterSubdivisionRemark} we have
	\benn
		\ba
			\int_{e} f(x) F(e) \frac{dx}{\ell(e)} &= 
				\int_{e'} f(x) \frac{\ell(e')}{\ell(e)} F(e) \frac{dx}{\ell(e')} 
				+ \int_{e''} f(x) \frac{\ell(e'')}{\ell(e)} F(e) \frac{dx}{\ell(e'')}  \\
			&= \int_{e'} f(x)  F(e') \frac{dx}{\ell(e')} 
				+ \int_{e''} f(x) F(e'') \frac{dx}{\ell(e'')}.
		\ea
	\eenn
It follows easily from this computation that the Zhang measure is independent of the model used to define it.

There is a description of $\mu_{\Zh}$ in terms of spanning trees of $G$ which is reminiscent of (\ref{eq: OmegaXDef}).
For each spanning tree $T$ of $G$, let $e_1^T, \ldots, e_g^T$ be the edges of $G$ lying outside of $T$.
It follows from Corollary~\ref{Cor: Foster-Trees} that
\begin{equation}
\tag{TC1}
\mu_{\Zh} =\sum_{k=1}^g \left( \sum_T \frac{w(T)}{w(G)} 
 \frac{dx|_{e_k^T}}{\ell(e_k^T)} \right).
\end{equation}

By Foster's theorem (Corollary~\ref{Thm: Foster}), the total mass of
$\mu_{\Zh}$ is equal to the genus of $\Gamma$:
\begin{equation}
\tag{TC2}
\int_\Gamma \mu_{\Zh} = g.
\end{equation}

Finally, let $J(\Gamma)$ be the Jacobian of $\Gamma$, and let
$\Phi_p : \Gamma \to J(\Gamma)$ be the Abel-Jacobi map relative to a fixed base point $p \in \Gamma$.
Recall that there is a canonical metric 
$\|\cdot\|_{\sup}$ on 
$J(\Gamma)$, called the Foster/Zhang metric,
coming from the identification of $J(\Gamma)$ with
the real torus $\Omega(\Gamma)^* / H_1(\Gamma,\ZZ)$.
If $e$ is an edge of the model $G$, then by Theorem~\ref{Thm: Foster Metric}
we have
\begin{equation}
\tag{TC3}
\| \Phi_p(e) \|_{\sup} = \mu_{\Zh}(e).
\end{equation}


\bibliographystyle{plain}
\bibliography{Tropical_Metric}

\end{document}